\documentclass[a4paper,reqno, 12pt]{amsart}
\usepackage[left=2.8cm,right=2.8cm,top=2.86cm,bottom=2.86cm]{geometry}
\usepackage{amsmath,amssymb,latexsym,esint,cite,mathrsfs}
\usepackage{amsmath}
\usepackage{amssymb}
\usepackage{latexsym}
\usepackage{amsthm}
\usepackage{hyperref}
\usepackage{array}
\usepackage{color}
\hypersetup{linkcolor=blue, colorlinks=true ,citecolor = red}
\DeclareMathOperator*{\essinf}{ess\,inf}
\DeclareMathOperator*{\esssup}{ess\,sup}
\DeclareMathOperator*{\loc}{loc}
\newcolumntype{M}[1]{>{\centering\arraybackslash}m{#1}}
\newcolumntype{N}{@{}m{0pt}@{}}
\newcommand*\diff{\mathop{}\!\mathrm{d}}
\begin{document}
	\title[$p(x)$-Laplacian on the whole space $\mathbb{R}^N$]
	{Existence results for Schr\"odinger  $p(x)$-Laplace equations involving critical growth in  $\mathbb{R}^N$}
	
	\author
	{Ky Ho, Yun-Ho Kim and Inbo Sim}
	\address{Ky Ho \newline
		Department of Mathematics, Faculty of Sciences, Nong Lam University, Linh Trung Ward, Thu Duc District, Ho Chi Minh City, Vietnam}
	\email{hnky81@gmail.com}
	\address{Yun-Ho Kim\newline
		Department of Mathematics
		Education, Sangmyung University,
		Seoul, 110-743, Republic of Korea}
	\email{kyh1213@smu.ac.kr}
	\address{Inbo Sim \newline
		Department of Mathematics, University of Ulsan, Ulsan 44610, Republic of Korea}
	\email{ibsim@ulsan.ac.kr}
	
	
	\begin{abstract}
		
		We establish some existence results for Schr\"odinger  $p(x)$-Laplace equations in 
		$\mathbb{R}^N$ with various potentials and critical growth of nonlinearity that may occur on some nonempty set, although not necessarily the whole space $\mathbb{R}^N$. The proofs are mainly based on concentration-compactness principles in a suitable weighted variable exponent Sobolev space and its imbeddings.

	\end{abstract}
	
	\maketitle
	\numberwithin{equation}{section}
	\newtheorem{theorem}{Theorem}[section]
	\newtheorem{lemma}[theorem]{Lemma}
	\newtheorem{proposition}[theorem]{Proposition}
	\newtheorem{corollary}[theorem]{Corollary}
	\newtheorem{definition}[theorem]{Definition}
	\newtheorem{example}[theorem]{Example}
	\newtheorem{remark}[theorem]{Remark}
	\allowdisplaybreaks
	
\section{Introduction}
In this paper, we study the existence and multiplicity of solutions to $p(x)$-Laplace equations involving critical growth in $\mathbb{R}^N$ of the form
\begin{equation}\label{eq}
-\operatorname{div}\left(| \nabla u| ^{p(x)-2}\nabla u\right)+V(x)|u|^{p(x)-2}u=\lambda a(x)|u|^{r(x)-2}u+\theta b(x)|u|^{q(x)-2}u, 
\end{equation}
where $p,r,q\in C_+(\mathbb{R}^N) :=\left\{ s\in C(\mathbb{R}^N) : 1<\underset{x\in\mathbb{R}^N}{\inf}s(x)\leq \underset{x\in\mathbb{R}^N}{\sup}s(x)<\infty \right\},$ such that  $p(x)<N,\ p(x)<q(x)\leq p^\ast(x):=\frac{Np(x)}{N-p(x)}$ for all $x\in \mathbb{R}^N;$ 
$V,a,b\in P_+(\mathbb{R}^N) := \{w : \mathbb{R}^N \to \mathbb{R}~ $ is measurable and positive a.e.$\};$ and $\lambda, \theta$ are positive real parameters.

The study of the $p(x)$-Laplacian,  $\operatorname{div}\left(| \nabla u| ^{p(x)-2}\nabla u\right),$ a natural extension of the $p$-Laplacian,  
$\operatorname{div}\left(| \nabla u| ^{p-2}\nabla u\right),$ which itself is also a natural extension of the Laplacian $(p=2),$
has been the center of attention recently since it has not only mathematical challenges (in particular, inhomogeneity) but also many applications, for examples, image processing \cite{Chen} and electrorheological fluids \cite{Ruzicka}. The Lebesgue-Sobolev spaces related to the $p(x)$-Laplacian are called variable exponent Lebesgue-Sobolev spaces and  were studied in \cite{Fan.MIA, Kovacik, Zhikov}. We also refer the reader to the book \cite{Diening} for basic knowledge in this area. 

Problem \eqref{eq} contains  the Schr\"odinger operator as a special case and can be reduced to  the time-independent modified Gross-Pitaevskii equation \cite{Gross, Pitaevskii } in the theory of Bose-Einstein condensates:
$$
-\operatorname{div}\left(\nabla u\right)+ V(x) u=\lambda a(x)|u|^{r-2}u+\theta b(x)|u|^{q-2}u\ \ ~\mbox{in}~\mathbb{R}^N,
$$
where 
$1<r <2^*, 2< q \le 2^*.$  We refer to Rabinowitz's seminal paper \cite{Ra}, which 
nicely explains nonlinear Schr\"odinger equations and has influenced many studies.
Since then there have been extensions of \cite{Ra} in many directions. Here we shall extend \cite{Ra} to Schr\"odinger  $p(x)$-Laplace equations with various potentials $V$. 

Fu-Zhang \cite{Fu-Zhang} considered 
\begin{equation}\label{fz}
-\operatorname{div}\left(| \nabla u| ^{p(x)-2}\nabla u\right)+ \lambda |u|^{p(x)-2}u= f(x,u) + h(x) |u|^{p^*(x)-2}u\ \  ~\mbox{in}~\mathbb{R}^N,
\end{equation}
where $f$ satisifies subcritical growth and 
 $h\in C(\mathbb{R}^N)\cap L^\infty(\mathbb{R}^N)$ is a nonnegative function. They obtained a concentration-compactness principle (for short, CCP) and applied this CCP to obtain the existence of solutions for \eqref{fz}. To avoid concentrations, they had to assume the radial symmetry of $f$ and $h$ with respect to the $x$ variable and that $h(0)=h(\infty)=0$. The CCP was originally provided by Lions \cite{Lions} for  critical problems  in a  bounded  domain  and  later  extended by  Chabrowski \cite{Chabrowski} for  critical problems in  an unbounded  domain. The variable exponents version for a bounded domain was independently obtained in \cite{Fu} and \cite{Bonder}.

Saintier-Silva \cite{Saintier} considered
\begin{equation}\label{ss}
-\operatorname{div}\left(| \nabla u| ^{p(x)-2}\nabla u\right)+ \lambda k(x) |u|^{p(x)-2}u= K(x) |u|^{q(x)-2}u\ \  ~\mbox{in}~\mathbb{R}^N,
\end{equation}
where $q(x)\leq p^*(x)$ for all $x\in\mathbb{R}^N$ with $p^+<q^-$ (see Section 2 for the definition) and assumed further that:
\begin{itemize}
	\item [$(SS1)$] $p$ and $q$ have a modulus of continuity $\rho$ in the sense that  for any $x, h \in \mathbb{R}^N,$ 
	$$p(x+h) =p(x) +\rho(h),~~\mbox{and}~~q(x+h) =q(x) +\rho(h)$$
	with $\lim_{h \to 0} \rho(h) \ln |h| =0;$
\end{itemize}
\begin{itemize}
	\item [$(SS2)$] there exist real numbers $p(\infty)$ and $q(\infty)$ with $q(\infty) = p^*(\infty)$ such that
	$$\lim_{|x| \to +\infty} |p(x) - p(\infty)| \ln |x| = 0~~\mbox{and}~~\lim_{|x| \to +\infty} |q(x) - q(\infty)| \ln |x| = 0;$$
\end{itemize}
\begin{itemize}
	\item [$(SS3)$] $k\in C(\mathbb{R}^N, [0,\infty))$ is bounded such that the imbedding from $\mathcal{D}^{1,p(x)}(\mathbb{R}^N)$ into $L^{p(x)}(k,\mathbb{R}^N)$ is compact, where $\mathcal{D}^{1,p(x)}(\mathbb{R}^N)$ is the completion of $C_c^\infty(\mathbb{R}^N)$ with respect to the norm $\|\nabla u\|_{L^{p(x)}(\mathbb{R}^N)}$;
\end{itemize}
\begin{itemize}
	\item [$(SS4)$] $K \in C(\mathbb{R}^N, [0,\infty))$ has a limit $K(\infty)$ (maynot necessarily 0) at infinity.
\end{itemize}
They first refined the CCP in \cite{Fu-Zhang} at infinity under the assumptions $(SS1)$ and $(SS2)$ and used this to provide a sufficient existence condition for problem \eqref{ss}. It is worth noting that they do not require that $K$ vanishes at  infinity, as $h$ does in  \cite{Fu-Zhang}, and also note that the critical set $\mathcal{A} := \{x \in \mathbb{R}^N : q(x)=p^\ast(x) \}$ may be empty or not. However, we believe that their existence result is incorrect, specifically with respect to the range of $\lambda$. This error stems from an incorrect assertion, which states that the compact imbedding  $\mathcal{D}^{1,p(x)}(\mathbb{R}^N)\hookrightarrow\hookrightarrow L^{p(x)}(k,\mathbb{R}^N)$ implies that the infimum $\lambda_*:=\underset{u\in\mathcal{D}^{1,p(x)}(\mathbb{R}^N)\setminus\{0\}}{\inf} \frac{\int_{\mathbb{R}^N}|\nabla u|^{p(x)}\diff x}{\int_{\mathbb{R}^N}|u|^{p(x)}k\diff x}$ is positive. It is known that in variable exponent Lebesgue-Sobolev spaces defined on a bounded main, such an infimum is zero if $p$ has a strict local mimimum (or maximum) (see \cite[Lemma 3.1 and Theorem 3.1]{Fan-Zhang-Zhao}). Similarly, we easily construct an example to show that $\lambda_*$ above can be zero. For instance, let $p\in C^1(\mathbb{R}^N)$ be such that $\underset{|x|\leq\frac{1}{4}}{\min}\ p(x)= p_0$ and $p(x)=p_1$ for $|x|\geq \frac{1}{2}$, where $p_0$ and $p_1$ are constants such that $1<p_0<p_1<N$ and let $k(x)=\frac{1}{\left(1+|x|\right)^\alpha}$ for some constant $\alpha\in (N,\infty).$ Then it is easy to verify that such $p$ and $k$ satisfy the conditions of \cite[Proposition 5.7]{Saintier} and hence, $\mathcal{D}^{1,p(x)}(\mathbb{R}^N)\hookrightarrow\hookrightarrow L^{p(x)}(k,\mathbb{R}^N)$ while $\lambda_*=0$, which can be easily shown as in the proof of \cite[Theorem 3.1]{Fan-Zhang-Zhao}.

We emphasize that the critical growth condition  $\mathcal{A}  \neq \emptyset $ distinguishs $p(x)$-Laplace problems from $p$-Laplace problems, which only allow  $\mathcal{A} = \mathbb{R}^N.$ Thus this condition is valuable for the study of  $p(x)$-Laplace problems. As far as the authors know, Mih\u{a}ilescu-R\u{a}dulescu \cite{MR}  first assumed   such a  characterizing condition to study the following problem with a bounded domain:   
\begin{eqnarray}\label{mr}
\begin{cases}
-\operatorname{div}\left(| \nabla u| ^{p(x)-2}\nabla u\right) = \lambda |u|^{q(x)-2}u  &\text{in}~ \Omega ,\\
u=0, &\text{on}~ \partial \Omega,
\end{cases}
\end{eqnarray}
where $p, q \in C(\overline{\Omega})$ and $q^- < p^- <q^+.$ The existence of a nontrivial solution was then obtained using the Ekeland variational method. Ho-Sim \cite{Ky1} generalized the condition $q^- < p^- <q^+$ to  $\{x \in \overline{\Omega} : q(x)< p(x) \} \neq \emptyset $ to study degenerate $p(x)$-Laplace equations involving concave-convex type nonlinearities  of  \eqref{eq}. The condition $\{x \in \mathbb{R}^N : q(x)=p^\ast(x) \} \neq \emptyset $ for a critical growth problem was first assumed in \cite{Bonder}. Another interesting aspect of critical $p(x)$-Laplace problems in $\mathbb{R}^N$ is the various conditions on the coefficient $V$ of the Schr\"odinger term. Except Saintier-Silva's result \cite{Saintier}, all existing results had to assume that  $\underset{x\in \mathbb{R}^N}{\essinf}\ V(x) >0.$ Even the potential $k$ in \cite{Saintier} with a parameter $\lambda>0$ does not require $\underset{x\in \mathbb{R}^N}{\essinf}\ k(x) >0,$ it seems to be nontrivial to find a $k$ which satisfies $(SS3)$. One of the novelties of this paper is that we assume $0<\underset{|x|<R}{\essinf}\ V(x)\leq\underset{|x|<R}{\esssup}\ V(x)<\infty$ for all $R>0$ and allow $\underset{x\in \mathbb{R}^N}{\essinf}\ V(x)$ to be zero by using only one additional natural assumption $|\{x\in\mathbb{R}^N:q(x)\ne p^\ast(x)\}|<\infty$, where $|\cdot|$ is the  Lebesgue measure in $\mathbb{R}^N.$ 
It is worth pointing out that we do not need to assume $q(\infty)=p^\ast(\infty)$ or some condition on the weights like symmetry or periodic or vanishing at infinity to deal with the possible loss of mass at infinity, as in existing results. Moreover, we shall also assume the concave term of the form $\{x \in \mathbb{R}^N : r(x)< p(x) \} \neq \emptyset $ for the existence result of  \eqref{eq}.

Pucci-Zhang \cite{Pucci-Zhang} considered 
\begin{equation}\label{pz}
-\operatorname{div}\left(| \nabla u| ^{p(x)-2}\nabla u\right)+V(x)|u|^{p(x)-2}u=\lambda a(x)|u|^{r(x)-2}u-h(x)|u|^{s(x)-2}u\ \  ~\mbox{in}~\mathbb{R}^N,
\end{equation}
where $\lambda$ is a parameter, $V \in L^{\infty}_{\rm{loc}}(\mathbb{R}^N),$ $V(x) \ge c (1+|x|)^{-p(x)}$ for all $x \in \mathbb{R}^N$ for some $c \in (0,1],$ $h\in P_+(\mathbb{R}^N),$ and $s\in C_+(\mathbb{R}^N)$ possibly supercritical.
Taking advantage of the  weighted variable exponent Sobolev space (which is similar to the approach in Section 2), as well as the effects of $a$ and $h,$ they obtained several existence results of nontrivial nonnegative solutions for problem \eqref{pz}.


The goal of this paper is to remove the restrictions on $p$ and $q$, such as $(SS1)$ and $(SS2)$, as well as on $b$ (which corresponds to $K$ in \eqref{ss}) such as $(SS4)$. We do this in order to study \eqref{eq} with various types of potentials $V$ by adopting a weighted variable exponent Sobolev space, as motivated in \cite{Pucci-Zhang} and establishing two CCPs on this space, as motivated in \cite{Fu-Zhang} and \cite{Saintier}. Note that we only assume the combined effects of $a$ and $b$ to obtain the existence result for \eqref{eq} when the nonlinearity is of the concave-convex type. Meanwhile we do not assume any combined conditions on $a$ and $b$ when the nonlinearity is of the $p(\cdot)$-superlinear type.

On the other hand, Bartsch-Wang  \cite{Bartsch} used such a weighted Sobolev space to study the Schr\"odinger equation with a coercive potential $V$, that is,  $V(x) \to \infty$ as $|x| \to \infty$ to overcome the lack of compactness. We realized that this type of potential also works well for establishing the CCPs on the weighted variable exponent Sobolev space. Thus, another novelty of our paper is that we consider Schr\"odinger  $p(x)$-Laplace equations in the whole space $\mathbb{R}^N$ with the several kinds of potentials $V$, including bounded, decaying and coercive potentials.

For the existence results, we consider two cases: one is  $p(\cdot)$-superlinear and the other is concave-convex in the nonlinearity term of \eqref{eq}. Garc\'{\i}a Azorero-Peral Alonso \cite{Garcia} initially studied problem \eqref{eq}  in a bounded domain, which is subject to  Dirichlet boundary conditions with the constant exponents and $V \equiv 0$, 
$a=b \equiv 1$ for a fixed  $\theta,$ {\it{i.e,}} $\theta \equiv 1.$ Ambrosetti-Brezis-Cerami \cite{Ambrosetti2} then provided a  nice study for the Laplacian when the nonlinearity is concave-convex, that is, $r<2<q$ and  $\theta \equiv 1.$ Since then there have been extensions of \cite{Garcia, Ambrosetti2} for problems with both constant exponents and variable exponents (for example, \cite{de Figueiredo, Bonder, Ky2}).
Among them, Ho-Sim \cite{Ky2} first considered degenerate $p(x)$-Laplacian problems  involving critical growth in a bounded domain with two parameters, as in \eqref{eq}, and found the range of an unfixed parameter to obtain the existence and multiplicity results of  \eqref{eq} for  the case of $p(\cdot)$-superlinear and concave-convex nonlinearity according to a fixed parameter. Thus this paper gives not only a unified way to deal with  the Schr\"odinger  $p(x)$-Laplace equation with various potentials under weaker assumptions on $p, q, b$, the characterizing condition  $\mathcal{A}  \neq \emptyset ,$ and $\{x \in \mathbb{R}^N : r(x)< p(x) \} \neq \emptyset $ for the concave term, but also new results for each of potentials above.

The paper is organized as follows. In Section 2, we define weighted variable exponent Lebesgue-Sobolev spaces and  list basic properties of them, which will be used in the next sections. In Section 3, we state and prove two concentration-compactness principles for weighted variable exponent spaces defined on $\mathbb{R}^N$, which are essential to overcome the lack of compactness of the Sobolev imbeddings in the whole space.
The last section is devoted to the existence and multiplicity of solutions to problem \eqref{eq} using the Mountain Pass Theorem, Ekeland variational principle and genus theory. In the Appendices, we give proofs of the Lemmas for imbeddings, which are crucial for obtaining useful inequalities.

\section{Abstract framework and preliminary results}\label{Preliminaries}
Let $\Omega$ be a  (bounded or unbounded) domain in $\mathbb{R}^N$.
For $w\in P_+(\Omega)$ and  $p\in C_{+}(\overline{\Omega}), $  we define
\[
L^{p(x) }(w,\Omega ) =\big\{  u:\Omega
\to\mathbb{R}\text{ is measurable, }\int_{\Omega }w(x)| u(x)| ^{p(x) }\diff x<\infty\big\} .
\]
Then, $L^{p(x) }(w,\Omega ) $ is a normed space with the norm
\[
| u| _{L^{p(x)}(w,\Omega )}=\inf \big\{ \lambda >0:\int_{\Omega }w(x)\Big| \frac{u(x) }{\lambda }
\Big| ^{p(x) }\diff x\leq 1\big\}.
\]
When $w(x)\equiv 1$, we use the notation $L^{p(x)}(\Omega )$ and $| u| _{L^{p(x)}(\Omega )}$ instead of $L^{p(x) }(w,\Omega )$ and $| u| _{L^{p(x)}(w,\Omega )}$, respectively. Set $p^{-}:=\inf_{x\in \overline{\Omega} } p(x) ,\  p^{+}:=\sup_{x\in \overline{\Omega}} p(x),$ and $L_+^{p(x)}(\Omega ):=L^{p(x)}(\Omega )\cap P_+(\Omega).$

We now state basic properties of weighted variable exponent Lebesgue spaces.

\begin{proposition}[\cite{Diening,Kovacik}]\label{2.prop1}
	The space $L^{p(x) }(\Omega )$ is a separable and uniformly convex Banach space, and its conjugate space is $L^{p'(x) }(\Omega ),$ where  $1/p(x)+1/p'(x)=1$. For any $u\in L^{p(x)}(\Omega)$ and $v\in L^{p'(x)}(\Omega)$, we have the following H\"older type inequality:
	\begin{equation*}
	\left|\int_\Omega uv\,\diff x\right|\leq\left(\frac{1}{p^-}+
	\frac{1}{{(p')}^-}\right)|u|_{L^{p(x)}(\Omega )}|v|_{L^{p'(x)}(\Omega )}\leq\ 2 |u|_{L^{p(x)}(\Omega )}|v|_{L^{p'(x)}(\Omega )}.
	\end{equation*}
\end{proposition}
Define the modular $\rho :L^{p(x) }(w,\Omega )$ $ \to \mathbb{R}$ as
\[
\rho (u) =\int_{\Omega }w(x)| u(x)| ^{p(x) }\diff x,\quad
\forall u\in L^{p(x) }(w,\Omega ) .
\]
The relationship between the modular and the corresponding norm is as follows.
\begin{proposition}[\cite{Kim}] \label{2.prop2}
	For all $u\in L^{p(x) }(w,\Omega ),$  we have
	\begin{itemize}
		\item[(i)] $|u|_{L^{p(x)}(w,\Omega )}<1$ $(=1,>1)$
		if and only if \  $\rho (u) <1$ $(=1,>1)$, respectively;
		
		\item[(ii)] If \  $|u|_{L^{p(x)}(w,\Omega )}>1,$ then  $|u|^{p^{-}}_{L^{p(x)}(w,\Omega )}\leq \rho (u) \leq |u|_{L^{p(x)}(w,\Omega )}^{p^{+}}$;
		\item[(iii)] If \ $|u|_{L^{p(x)}(w,\Omega )}<1,$ then $|u|_{L^{p(x)}(w,\Omega )}^{p^{+}}\leq \rho
		(u) \leq |u|_{L^{p(x)}(w,\Omega )}^{p^{-}}$.
	\end{itemize}
	Consequently,
	$$|u|_{L^{p(x)}(w,\Omega )}^{p^{-}}-1\leq \rho (u) \leq |u|_{L^{p(x)}(w,\Omega )}^{p^{+}}+1,\ \forall u\in L^{p(x) }(w,\Omega ).$$
\end{proposition}
Thus, modular convergence is equivalent to norm convergence.
\begin{proposition} \label{2.prop3}
	If $u,u_n\in L^{p(x) }(w,\Omega )~  (n=1,2,\cdots),$ then the
	following statements are equivalent:
	\begin{itemize}
		\item[(i)] $\lim_{n\to \infty }|u_n-u|_{L^{p(x)}(w,\Omega )}=0$;
		
		\item[(ii)] $\lim_{n\to \infty }\rho (u_n-u)=0$.
		
	\end{itemize}
\end{proposition}
We  also need the following relations of norms in different variable exponent spaces.
\begin{proposition} [\cite{Sim-Kim}] \label{2.Sim-Kim}
	Let $p\in C_{+}(\overline{\Omega })$ and $q\in C(\overline{\Omega })$ such that $pq\in C_{+}(\overline{\Omega })$. For all $u\in L^{p(x)q(x) }(w,\Omega ),$  we have
	\begin{itemize}
		
		\item[(i)] If  $|u|_{L^{p(x)q(x) }(w,\Omega )}\geq 1,$ then  $|u|_{L^{p(x)q(x) }(w,\Omega )}^{q^-}\leq \big||u|^q\big|_{L^{p(x) }(w,\Omega )}\leq |u|_{L^{p(x)q(x) }(w,\Omega )}^{q^+}$;
		\item[(ii)] If  $|u|_{L^{p(x)q(x) }(w,\Omega )}<1,$ then  $|u|_{L^{p(x)q(x) }(w,\Omega )}^{q^+}\leq \big||u|^q\big|_{L^{p(x) }(w,\Omega )}\leq |u|_{L^{p(x)q(x) }(w,\Omega )}^{q^-}$.
	\end{itemize}
	Consequently, we have
	$$|u|_{L^{p(x)q(x) }(w,\Omega )}^{q^{-}}-1\leq \big||u|^q\big|_{L^{p(x) }(w,\Omega )}\leq |u|_{L^{p(x)q(x) }(w,\Omega )}^{q^{+}}+1.$$
\end{proposition}

\begin{remark}\label{2.measureable.spaces}
	The above properties of the norm and the modular are still true for any $L_\mu^{p(x) }(\Omega ) =\{  u:\Omega
	\to\mathbb{R}\text{ is $\mu$-measurable, }\int_{\Omega }| u(x)| ^{p(x) }\diff\mu<\infty \}$ with corresponding modular $\rho(u)=\int_{\Omega }| u(x)| ^{p(x) }\diff\mu,$ where $\mu$ is any measure on $\Omega$ (see \cite{Diening}).
\end{remark}
For $V\in P_+(\Omega)$ and  $p\in C_{+}(\overline{\Omega}), $  the weighted variable exponent Sobolev space $W_V^{1,p(x)}(\Omega) $ is defined by
\[
W_V^{1,p(x)}(\Omega ) =\{u\in L^{p(x) }(V,\Omega) :
|\nabla u|\in L^{p(x) }(\Omega ) \},
\]
with the norm
\[
\|u\|_{W_V^{1,p(x)}(\Omega )}=|u|_{L^{p(x)}(V,\Omega )}+\big||\nabla u|\big|_{L^{p(x)}(\Omega )}.
\]
As before, when $V(x)\equiv 1$, we write $W^{1,p(x)}(\Omega )$ and  $\|u\|_{W^{1,p(x)}(\Omega )}$ in place of $W_V^{1,p(x)}(\Omega )$ and $\|u\|_{W_V^{1,p(x)}(\Omega )}$, respectively. 

In the case of a bounded Lipschitz domain $\Omega$, it is well known that $W^{1,p(x)}(\Omega)\hookrightarrow\hookrightarrow L^{q(x)}(\Omega)$ for any $q\in C_+(\overline{\Omega})$ with $q(x) < p^{\ast }(x)$ for all $x\in\overline{\Omega}.$ Moreover, if $p^+<N$ and $p$ is log-H\"older continuous in $\overline{\Omega}$, namely,
$$-|p(x)-p(y)|\log|x-y|\leq C,\quad \forall x,y\in\overline{\Omega}\ \ \text{with}\ \ |x-y|<\frac{1}{2},$$
then we have the critical imbedding $W^{1,p(x)}(\Omega) \hookrightarrow L^{p^\ast(x) }(\Omega )$ (see \cite{Diening}). 

For an unbounded domain with no weights, we recall the following imbedding.
\begin{proposition}[\cite{Fan.MIA}]\label{2.prop.unbounded}
	Let $\Omega\subset\mathbb{R}^N$ be an unbounded domain satisfying the cone uniform condition. Suppose that $p\in C_+(\overline{\Omega})$ is Lipschitz in $\overline{\Omega}.$ Then there holds a continuous imbedding $W^{1,p(x)}(\Omega) \hookrightarrow  L^{q(x) }(\Omega ),$ for any $q\in L^\infty(\Omega)$ satisfying condition
	$$p(x)\leq q(x)\leq p^\ast(x)\ \text{for a.e.}\ x\in\Omega.$$
\end{proposition}

\section{Concentration-Compactness Principles}
\subsection{Statements of the concentration-compactness principles} In this section, we establish two concentration-compactness principles for weighted variable exponent spaces defined on $\mathbb{R}^N$. These results are generalizations of several known results for non-weighted variable exponent spaces.

Let $C_c(\mathbb{R}^N)$ be the
set of all continuous functions $u : \mathbb{R}^N \to\mathbb{R}$ whose support is compact, and let $C_0(\mathbb{R}^N)$ be the completion of $C_c(\mathbb{R}^N)$ relative to the supremum norm $\|\cdot\|_\infty.$ Let $\mathcal{M}(\mathbb{R}^N)$ be the space of all signed finite
Radon measures on $\mathbb{R}^N$ with the total variation norm. We may identify $\mathcal{M}(\mathbb{R}^N)$ with the dual of $C_0(\mathbb{R}^N)$ via the Riesz representation theorem, that is, for each $\mu\in \left[C_0(\mathbb{R}^N)\right]^\ast$ there is a unique element in $\mathcal{M}(\mathbb{R}^N)$, still denoted by $\mu$, such that
$$\langle \mu,f\rangle=\int_{\mathbb{R}^N}fd\mu,\quad \forall f\in C_0(\mathbb{R}^N)$$
(see, e.g., \cite[Section 1.3.3]{Fonseca}). We identify $L^1(\mathbb{R}^N)$ with a subspace of $\mathcal{M}(\mathbb{R}^N)$ through the imbedding
$T:\ L^1(\mathbb{R}^N)\to \left[C_0(\mathbb{R}^N)\right]^\ast$  defined by
$$\langle Tu,f\rangle=\int_{\mathbb{R}^N}uf\diff x, \ \ \forall u\in L^1(\mathbb{R}^N), \ \forall f\in C_0(\mathbb{R}^N).$$
In what follows, the notation $u_n\to u$ (resp. $u_n \rightharpoonup u,u_n \overset{\ast }{\rightharpoonup }u$) indicates for the term $u_n$ strongly (resp. weakly, weakly-$\ast$) converges to $u$ as $n \to \infty$ in an appropriate space. The notations $B_\epsilon(x),$ $|E|$ and $E^c$ stand for an open ball in $\mathbb{R}^N$ centered at $x$ with radius $\epsilon$, the Lebesgue measure, and the complement of a set $E$ in $\mathbb{R}^N$, respectively. Unless stated otherwise, we write $B_\epsilon$ in place of $B_\epsilon(0).$ We also denote the characteristic function of $E$ by $\chi_E.$

Throughout this paper, we assume for the exponent $p$ and the potential term $V$ that
\begin{itemize}
	\item[$(\mathcal{P})$] $p: \mathbb{R}^N\to \mathbb{R}$ is Lipschitz continuous such that $1<p^-\leq p^+<N$; $0<\underset{x\in B_R}{\essinf}\ V(x)\leq\underset{x\in B_R}{\esssup}\ V(x)<\infty$ for all $R>0$. 
	Furthermore, in the case $\underset{x\in\mathbb{R}^N}{\essinf}\ V(x)=0,$ we assume in addition that $p$ satisfies the log-H\"older decay condition, namely, that there exist constants $p_\infty$ and $C$ such that
	$$|p(x)-p_\infty|\leq \frac{C}{\log (e+|x|)},\quad \forall x\in\mathbb{R}^N.$$ 
\end{itemize}
Under the assumption $(\mathcal{P}),$ we assume that the exponent $q$ and the weight $b$ satisfy
\begin{itemize}
	\item[$(\mathcal{C})$] $q\in C(\mathbb{R}^N)$, $p(x)<q(x)\leq p^\ast(x)$ for all $x\in\mathbb{R}^N$ with $\underset{x\in\mathbb{R}^N}{\inf}(q(x)-p(x))>0$, there is a bounded Lipschitz (nonempty) domain $\Omega$ of $\mathbb{R}^N$ such that
	$\emptyset \ne \mathcal{A}:=\{x\in \mathbb{R}^N: q(x)=p^{\ast }(x)\}\subset (\overline{\Omega})^c;$ furthermore, in the case $\underset{x\in\mathbb{R}^N}{\essinf}\ V(x)=0$ we additionally assume that $|\{x\in\mathbb{R}^N:q(x)\ne p^\ast(x)\}|<\infty;$ 
	$b\in P_+(\mathbb{R}^N)$, $0<b(x)\leq b_0$ a.e. in $\Omega^c$ for some positive constant $b_0,$ and	$b\in L^{\beta(x)}(\Omega)$ for some $\beta\in C_+(\overline{\Omega})$ satisfying  $q(x)<\frac{\beta(x)-1}{\beta(x)}p^\ast(x)$ for all $x\in \overline{\Omega}.$
\end{itemize}
We emphasize that for simplicity and clarity of the assumptions, we only stated them for $\Omega\ne\emptyset,$ but in the sequel, we shall also include the case $\Omega=\emptyset.$ In this case, $(\mathcal{C})$ reads
\begin{itemize}
	\item[$(\mathcal{C})$] $q\in C(\mathbb{R}^N)$, $q(x)\leq p^\ast(x)$ for all $x\in\mathbb{R}^N$ with $\underset{x\in\mathbb{R}^N}{\inf}(q(x)-p(x))>0$ and
	$\mathcal{A}:=\{x\in \mathbb{R}^N: q(x)= p^{\ast }(x)\}\ne\emptyset.$ Furthermore in the case $\underset{x\in\mathbb{R}^N}{\essinf}\ V(x)=0$ we assume in addition that $|\{x\in\mathbb{R}^N:q(x)\ne p^\ast(x)\}|<\infty;$
	$b\in P_+(\mathbb{R}^N)$ satisfying $0<b(x)\leq b_0$ a.e. in $\mathbb{R}^N$ for some positive constant $b_0.$
\end{itemize}
It is worth mentioning that the potential $V(x)$ may possibly tend to zero or infinity as $|x|\to\infty$. Some examples for such potentials are $V(x)=1+|x|$ and $V(x)=\frac{c}{(1+|x|)^{p(x)}}$. By standard arguments, we can verify that $W_V^{1,p(x) }(\mathbb{R}^N) $ is a reflexive separable Banach space under the assumption $(\mathcal{P})$ (cf. \cite{Diening, Kim, Ky3}). Define $X$ as the closure of $C_c^\infty(\mathbb{R}^N)$ in $W_V^{1,p(x)}(\mathbb{R}^N)$, and for brevity we write $\|\cdot\|$ instead of $ \|\cdot\|_{W_V^{1,p(x)}(\mathbb{R}^N)}$ in what follows. 
We have the following crucial imbeddings.

\begin{lemma} \label{III.critical.imb.loc}
	Assume that $0<\underset{x\in B_R}{\essinf}\ V(x)\leq\underset{x\in B_R}{\esssup}\ V(x)<\infty$ for all $R>0$. Then, 
	\begin{center}
		$X \hookrightarrow W_V^{1,p(x)}\left(\mathbb{R}^N\right) \hookrightarrow \hookrightarrow L_{\loc}^{p(x) }\left(\mathbb{R}^N \right).$
	\end{center}
\end{lemma}

\begin{lemma} \label{III.critical.imb}
	Assume that $(\mathcal{P})$ and $(\mathcal{C})$ hold. Then, we have the following continuous imbedding
	\begin{center}
		$X \hookrightarrow L^{q(x) }\left(b,\mathbb{R}^N \right).$
	\end{center}
In particular, under assumption $(\mathcal{P})$ we have the following continuous imbedding
\begin{center}
	$X \hookrightarrow L^{p^\ast(x) }\left(\mathbb{R}^N \right).$	
\end{center}
\end{lemma}

The proofs of Lemmas \ref{III.critical.imb.loc} and \ref{III.critical.imb} are provided in Appendix~\ref{AppendixA}.
Note that by Lemma~\ref{III.critical.imb}, we have
\begin{equation}\label{3.best.const}
S:=\underset{\phi\in X\setminus \{0\}}{\inf}
\frac{\|\phi\|}{| \phi |_{L^{q(x)}(b,\mathbb{R}^N)}}>0.
\end{equation}

The following concentration-compactness principles extend the results by Fu-Zhang \cite{Fu-Zhang} to the weighted space case.


\begin{theorem}\label{III.ccp}
	Assume that $(\mathcal{P})$ and $(\mathcal{C})$ hold. Let $\{u_n\}_{n\in\mathbb{N}}$ be a bounded sequence in
	$X$ such that
	\begin{gather*}
		u_n \rightharpoonup u \quad \text{in}\quad  X,\\
		|\nabla u_n|^{p(x)}+V|u_n|^{p(x)} \overset{\ast }{\rightharpoonup }\mu\quad \text{in}\quad \mathcal{M}(\mathbb{R}^N),\\
		b|u_n|^{q(x)}\overset{\ast }{\rightharpoonup }\nu\quad \text{in}\quad \mathcal{M}(\mathbb{R}^N).
	\end{gather*}
	Then, there exist $\{x_i\}_{i\in I}\subset \mathcal{A}$ of distinct points and $\{\nu_i\}_{i\in I}, \{\mu_i\}_{i\in I}\subset (0,\infty),$ where $I$ is at most countable, such that
	\begin{gather}
	\nu=b|u|^{q(x)} + \sum_{i\in I}\nu_i\delta_{x_i},\label{3.ccp.form.nu}\\
	\mu \geq |\nabla u|^{p(x)}+V|u|^{p(x)} + \sum_{i\in I} \mu_i \delta_{x_i},\label{3.ccp.form.mu}\\
	S \nu_i^{1/q(x_i)} \leq \mu_i^{1/p(x_i)}, \quad \forall i\in I.\label{3.ccp.nu_mu}
	\end{gather}
\end{theorem}
The information about possible loss of mass at infinity is provided in the next theorem.
\begin{theorem}\label{III.ccp.infinity}
	Assume that $(\mathcal{P})$ and $(\mathcal{C})$  hold. Let $u_n \rightharpoonup u$ in $X$ and set
	$$\mu_\infty:=\lim_{R\to\infty}\underset{n\to\infty}{\lim\sup}\int_{\{|x|>R\}}\left[|\nabla u_n|^{p(x)}+V(x)|u_n|^{p(x)}\right]\diff x,$$
	$$\nu_\infty:=\lim_{R\to\infty}\underset{n\to\infty}{\lim\sup}\int_{\{|x|>R\}}b(x)|u_n|^{q(x)}\diff x.$$
Then 	
	\begin{eqnarray}
		\underset{n\to\infty}{\lim\sup}\int_{\mathbb{R}^N}\left[|\nabla u_n|^{p(x)}+V(x)|u_n|^{p(x)}\right]\diff x=\mu(\mathbb{R}^N)+\mu_\infty,\label{3.ccp.infinity.mu}\\
	\underset{n\to\infty}{\lim\sup}\int_{\mathbb{R}^N}b(x)|u_n|^{q(x)}\diff x=\nu(\mathbb{R}^N)+\nu_\infty.\label{3.ccp.infinity.nu}
	\end{eqnarray}
	Assume in addition that:
	\begin{itemize}
		\item[$(\mathcal{E}_\infty)$] There exist positive real numbers $p_\infty$ and $q_\infty$ such that $\lim_{|x|\to\infty}p(x)=p_\infty$ and $\lim_{|x|\to\infty}q(x)=q_\infty.$
	\end{itemize}
	 Then
	\begin{equation}\label{3.ccp.infinity.nu_mu}
		\frac{1}{2}S \nu_\infty^{1/q_\infty} \leq \mu_\infty^{1/p_\infty}.
			\end{equation}
\end{theorem}

\begin{remark}
	
	It is worth mentioning that in the case $\underset{x\in\mathbb{R}^N}{\essinf}\ V(x)>0$ then we can replace $X$ with $W_V^{1,p(x)}(\mathbb{R}^N)$ in the definition \eqref{3.best.const}, Theorems \ref{III.ccp} and \ref{III.ccp.infinity} 
	and hence, when $\Omega=\emptyset$ and $V=b\equiv 1$ in $\mathbb{R}^N,$ we recover the concentration-compactness principles by Fu-Zhang \cite{Fu-Zhang}. Furthermore, we provide the relation between the weights of ``the Dirac masses at infinity" that was not given in \cite{Fu-Zhang} but discussed in \cite{Saintier}.
\end{remark}

\subsection{Proofs of the concentration-compactness principles } We shall provide the proofs of Theorems~\ref{III.ccp} and \ref{III.ccp.infinity} after obtaining several auxiliary results. Theorem~\ref{III.ccp} will be proved in a similar fashion to that of \cite[Theorem 3.3]{Ky2} by employing the following lemmas.


\begin{lemma}\label{III.lemma.convergence}
	Let $\nu,\{\nu_n\}_{n\in\mathbb{N}}$ be nonnegative and finite Radon measures on $\mathbb{R}^N$ such that $\nu_n\overset{\ast }{\rightharpoonup } \nu$ in $\mathcal{M}(\mathbb{R}^N)$. Then, for any $r\in C_{+}(\mathbb{R}^N)$,
	$$
	|\phi|_{L^{r(x)}_{\nu_n}(\mathbb{R}^N)} \to
	|\phi|_{L^{r(x)}_{\nu}(\mathbb{R}^N)}, \quad \forall \phi\in C_c(\mathbb{R}^N).$$
\end{lemma}
The proof is referred to \cite[Proof of Lemma 3.1]{Bonder} by replacing $\Omega$ with $\mathbb{R}^N.$


\begin{lemma}\label{III.lemma.reserveHolder}
	Let $\mu,\nu$ be two nonnegative and finite Radon measures
	on $\mathbb{R}^N$, such that there exists some  constant $C>0$ such that
	\begin{equation*}
	|\phi|_{L_\nu^{r(x)}(\mathbb{R}^N)}\leq C|\phi|_{L_\mu^{p(x)}(\mathbb{R}^N)},\ \ \forall \phi\in C_c^\infty(\mathbb{R}^N)
	\end{equation*}
	for some $p,r\in C_+(\mathbb{R}^N)$ satisfying $\underset{x\in\mathbb{R}^N}{\inf}(r(x)-p(x))>0.$ Then, there exist an at most countable set $\{x_i\}_{i\in I}$ of distinct points in $\mathbb{R}^N$ and
	$\{\nu_i\}_{i\in I}\subset (0,\infty)$ such that
	$$
	\nu=\sum_{i\in I}\nu_i\delta_{x_i}.
	$$
\end{lemma}
Since the proof of Lemma~\ref{III.lemma.reserveHolder} can be obtained by using a similar argument to that of \cite[Proof of Lemma 3.2]{Bonder}, taking into account the following lemma,  we omit it.

\begin{lemma}\label{III.le.est.measure}
	Let $\nu$ be a nonnegative and finite Radon measure
	on $\mathbb{R}^N$, such that there exists some  constant $C>0$ such that
	\begin{equation}\label{3.reverse.ineq}
	|\phi|_{L_\nu^{r(x)}(\mathbb{R}^N)}\leq C|\phi|_{L_\nu^{p(x)}(\mathbb{R}^N)},\ \ \forall \phi\in C_c^\infty(\mathbb{R}^N)
	\end{equation}
	for some $p,r\in C_+(\mathbb{R}^N)$ satisfying $\underset{x\in\mathbb{R}^N}{\inf}(r(x)-p(x))>0.$ 
	Then $\nu=0$ or there exist $\{x_i\}_{i=1}^n$ of distinct points in $\mathbb{R}^N$ and
	$\{\nu_i\}_{i=1}^n\subset (0,\infty)$ such that
	$$
	\nu=\sum_{i=1}^n\nu_i\delta_{x_i}.
	$$
\end{lemma}
We provide a proof of this lemma in Appendix~\ref{AppendixB}.
The next lemma generalizes the Brezis-Lieb Lemma and the proof can  easily be modified as in the constant case (see e.g., \cite[Proof of Lemma 1.32]{Willem}).


\begin{lemma}\label{III.brezis-lieb}
	Let $\{f_n\}$ be a bounded sequence in $L^{r(x)}(m, \mathbb{R}^N)$ ($r\in C_{+}(\mathbb{R}^N)$, $m\in P_+(\mathbb{R}^N)$) and $f_n(x)\to f(x)$ for a.e. $x\in \mathbb{R}^N$. Then $f\in L^{r(x)}(m, \mathbb{R}^N)$ and
	$$
	\lim_{n\to\infty}\int_{ \mathbb{R}^N} \left|m|f_n|^{r(x)}
	-m|f_n-f|^{r(x)}-m|f|^{r(x)}\right|\diff x=0.
	$$
\end{lemma}
We are now in a position to give a proof of Theorem~\ref{III.ccp}. The following proof is carried out by the same scheme as in \cite[Proof of Theorem 3.3]{Ky2}.


\begin{proof}[\textbf{Proof of Theorem~\ref{III.ccp}}]
	
We only prove the case for $\Omega\ne\emptyset$ since the proof for $\Omega=\emptyset$ is similar and actually simpler.	We first obtain \eqref {3.ccp.form.nu}. For this, let $v_n=u_n-u$. Then, in view of Lemma~\ref{III.critical.imb.loc}, we have up to a subsequence
	\begin{eqnarray}\label{3.conv.of.v_n}
	\begin{cases}
	v_n(x) &\to  \quad 0 \quad \text{for a.e.}\quad  x\in\mathbb{R}^N,\\
	v_n &\rightharpoonup \quad 0 \quad \text{in}\quad  X.
	\end{cases}
	\end{eqnarray}
	Thus, invoking Lemma~\ref{III.brezis-lieb}, we get
	$$
	\lim_{n\to\infty}\int_{\mathbb{R}^N}\left|b|u_n|^{q(x)}
	-b|v_n|^{q(x)}-b|u|^{q(x)}\right|\diff x=0.$$
	Combining this with \eqref{3.conv.of.v_n}, we deduce
	$$
	\lim_{n\to\infty}\Big(\int_{\mathbb{R}^N} \phi b|u_n|^{q(x)}\diff x
	-\int_{\mathbb{R}^N}\phi b|v_n|^{q(x)} \diff x\Big)
	=\int_{\mathbb{R}^N} \phi b|u|^{q(x)} \diff x, \quad \forall  \phi\in C_0(\mathbb{R}^N),$$
	i.e.,
	\begin{equation}\label{3.w*-vn}
	b|v_n|^{q(x)}\overset{\ast }{\rightharpoonup }\bar{\nu}=\nu-b|u|^{q(x)}\quad \text{in} \ \ \mathcal{M}(\mathbb{R}^N).
	\end{equation}
	Obviously, $\{|\nabla v_n|^{p(x)}+V|v_n|^{p(x)}\}$ is bounded in $L^1(\mathbb{R}^N)$. So up to a subsequence, we have
	\begin{equation}
	|\nabla v_n|^{p(x)}+V|v_n|^{p(x)} \overset{\ast }{\rightharpoonup }\quad \bar{\mu}\quad \text{in}\quad \mathcal{M}(\mathbb{R}^N),
	\end{equation}
	for some finite nonnegative Radon measure $\bar{\mu}$ on $\mathbb{R}^N$. Let  $\phi\in C_c^\infty(\mathbb{R}^N).$ Applying \eqref{3.best.const} we have
	\begin{align*}
	S|\phi v_n|_{L^{q(x)}(b,\mathbb{R}^N)} &\leq |\phi v_n|_{L^{p(x)}(V,\mathbb{R}^N)}+\big||\nabla(\phi v_n)|\big|_{L^{p(x)}(\mathbb{R}^N)}\notag\\
	&\leq |\phi v_n|_{L^{p(x)}(V,\mathbb{R}^N)}+\big||\phi\nabla v_n|\big|_{L^{p(x)}(\mathbb{R}^N)}+\big||v_n\nabla\phi|\big|_{L^{p(x)}(\mathbb{R}^N)}.
	\end{align*}
	Set $\bar{\nu}_n:=b|v_n|^{q(x)}$ and $\bar{\mu}_n:=|\nabla v_n|^{p(x)}+V|v_n|^{p(x)}$ and fix $R>0$ such that 
	$\operatorname{supp}(\phi)\subset B_R.$ Then we get from the last estimate that
	\begin{equation}\label{3.proofccp.est.norm1}
	S|\phi|_{L_{\bar{\nu}_n}^{q(x)}(\mathbb{R}^N)} \leq |\phi|_{L_{\bar{\mu}_n}^{p(x)}(\mathbb{R}^N)}+C(\phi,R)|v_n|_{L^{p(x)}(B_R)},
	\end{equation}
	where $C(\phi,R):=|\phi|_{L^\infty(\mathbb{R}^N)}\left(|V|^{\frac{1}{p^+}}_{L^\infty(B_R)}+|V|^{\frac{1}{p^-}}_{L^\infty(B_R)}\right)+\big||\nabla\phi|\big|_{L^\infty(\mathbb{R}^N)}.$
		Since $v_n \rightharpoonup 0$ in $X,$ we infer $v_n \to 0$ in space $L^{p(x)}(B_R)$ in view of Lemma~\ref{III.critical.imb.loc}. From this and invoking Lemma~\ref{III.lemma.convergence}, we deduce from \eqref{3.proofccp.est.norm1} that
	\begin{equation}\label{3.proofccp.est.norm2}
	S|\phi|_{L_{\bar{\nu}}^{q(x)}(\mathbb{R}^N)} \leq |\phi|_{L_{\bar{\mu}}^{p(x)}(\mathbb{R}^N)}.
	\end{equation}
	Since \eqref{3.proofccp.est.norm2} holds for any $\phi\in C_c^\infty(\mathbb{R}^N),$ we deduce from Lemma~\ref{III.lemma.reserveHolder} that there exist an at most countable set $\{x_i\}_{i\in I}$ of distinct points in $\mathbb{R}^N$ and
	$\{\nu_i\}_{i\in I}\subset (0,\infty)$ such that
	$$
	\bar{\nu}=\sum_{i\in I}\nu_i\delta_{x_i}.
	$$
	That is, we have just obtained \eqref{3.ccp.form.nu}.
	
	Next, we show that $\{x_i\}_{i\in I}\subset \mathcal{A}$. Assume on the contrary that there is some $x_i\in \mathbb{R}^N\setminus\mathcal{A}$. Let $\delta>0$ be such that $\overline{B_{2\delta}(x_i)}\subset \mathbb{R}^N\setminus\mathcal{A},$ noting the closedness of $\mathcal{A}$.  Set $B:=B_\delta(x_i)$ then $\overline{B}\subset \mathbb{R}^N\setminus\mathcal{A}.$ We claim that
	\begin{equation}\label{3.proofccp.lim.un}
	\lim_{n\to\infty}\int_Bb|u_n-u|^{q(x)}\diff x=0.
	\end{equation}
	Indeed, we have that $u_n \rightharpoonup u$ in $W_V^{1,p(x)}(B).$ Since  $\overline{B}\subset\mathbb{R}^N\setminus\mathcal{A},$ we have
	$q(x)<p^\ast(x)$ for all $x\in\overline{B}$ and hence, we have if $B\cap \Omega=\emptyset,$ then
	
	\begin{equation}\label{imb.B}
	W_V^{1,p(x)}(B)=W^{1,p(x)}(B)\hookrightarrow\hookrightarrow L^{q(x)}(B)= L^{q(x)}(b,B)
	\end{equation} and hence \eqref{3.proofccp.lim.un} holds. If $B\cap \Omega\ne\emptyset$, then we find $\widetilde{\beta}\in C_+(\overline{B})$ such that $\widetilde{\beta}|_{\overline{\Omega\cap B}}=\beta$ and $q(x)<\frac{\widetilde{\beta}(x)-1}{\widetilde{\beta}(x)}p^\ast(x)$ for all $x\in \overline{B}.$ 
	 Clearly, $b\in L^{\widetilde{\beta}(x)}(B)$ and thus we obtain
	$$W_V^{1,p(x)}(B)=W^{1,p(x)}(B)\hookrightarrow\hookrightarrow L^{q(x)}(b,B)$$
in view of \cite[Theorem 2.1]{Fan.JMAA2005}; hence, we deduce \eqref{3.proofccp.lim.un}.
	Therefore, 
	$$\int_{B}b|u_n|^{q(x)}\diff x\to \int_{B}b|u|^{q(x)}\diff x,$$
	in view of Lemma~\ref{III.brezis-lieb}.	From this and the fact that $\nu (B)\leq \liminf_{n\to \infty}\int_{B}b|u_n|^{q(x)}\diff x$ (see \cite[Proposition 1.203]{Fonseca}), we obtain $\nu (B)\leq \int_{B}b|u|^{q(x)}\diff x.$
	Meanwhile, from \eqref{3.ccp.form.nu}, we have
	$$\nu (B)\geq \int_{B}b|u|^{q(x)}\diff x+\nu_i>\int_{B}b|u|^{q(x)}\diff x,$$
	a contradiction. So $\{x_i\}_{i\in I}\subset \mathcal{A}$.

	Last, we obtain \eqref{3.ccp.nu_mu} and \eqref{3.ccp.form.mu}. Let  $\phi\in C_c^\infty(\mathbb{R}^N)$ be arbitrary.  By \eqref{3.best.const} we have
	\begin{align*}
	S|\phi u_n|_{L^{q(x)}(b,\mathbb{R}^N)} &\leq |\phi u_n|_{L^{p(x)}(V,\mathbb{R}^N)}+\big||\nabla(\phi u_n)|\big|_{L^{p(x)}(\mathbb{R}^N)}\\
	&\leq |\phi u_n|_{L^{p(x)}(V,\mathbb{R}^N)}+\big||\phi\nabla u_n|\big|_{L^{p(x)}(\mathbb{R}^N)}+\big||u_n\nabla\phi |\big|_{L^{p(x)}(\mathbb{R}^N)}\\
	&\leq\big||\phi\nabla u_n|\big|_{L^{p(x)}(\mathbb{R}^N)}+|\phi u|_{L^{p(x)}(V,\mathbb{R}^N)}+\big||u\nabla\phi|\big|_{L^{p(x)}(\mathbb{R}^N)}\\
	&{}\quad +|\phi( u_n-u)|_{L^{p(x)}(V,\mathbb{R}^N)}+\big||(u_n-u)\nabla\phi|\big|_{L^{p(x)}(\mathbb{R}^N)}.
	\end{align*}
	Hence,
	\begin{align*}\label{3.proofccp.est.un}
	S|\phi u_n|_{L^{q(x)}(b,\mathbb{R}^N)} &\leq\big||\phi\nabla u_n|\big|_{L^{p(x)}(\mathbb{R}^N)}+|\phi u|_{L^{p(x)}(V,\mathbb{R}^N)}+\big||u\nabla\phi|\big|_{L^{p(x)}(\mathbb{R}^N)}\notag\\
	&{}\quad + C(\phi,R')|u_n-u|_{L^{p(x)}(B_{R^{'}})}, 
	\end{align*}
	where $R^{'}>0$ is taken such that $\operatorname{supp}(\phi)\subset B_{R^{'}}$ and $$C(\phi,R'):=|\phi|_{L^\infty(\mathbb{R}^N)}\left(|V|^{\frac{1}{p^+}}_{L^\infty(B_{R'})}+|V|^{\frac{1}{p^-}}_{L^\infty(B_{R'})}\right)+\big||\nabla\phi|\big|_{L^\infty(\mathbb{R}^N)}.$$ Letting $n\to\infty$ in the last estimate with the same argument as that for \eqref{3.proofccp.est.norm2}, we have
	\begin{equation}\label{3.proofccp.estforsingular}
	S|\phi|_{L_{\nu}^{q(x)}(\mathbb{R}^N)} \leq |\phi|_{L_{\mu}^{p(x)}(\mathbb{R}^N)}+|\phi u|_{L^{p(x)}(V,\mathbb{R}^N)}+\big||u\nabla\phi|\big|_{L^{p(x)}(\mathbb{R}^N)}.
	\end{equation}
	Let $\eta$ be in $C_c^\infty(\mathbb{R}^N)$ such that $\chi_{B_{\frac{1}{2}}}\leq \eta\leq \chi_{B_1}.$  Suppose that $I\ne \emptyset$ and fix $i\in I.$ For $\epsilon>0$, set  $\phi_{i,\epsilon}(x):=
	\eta(\frac{x-x_i}{\epsilon})$. By the same argument as \cite[Proof of Theorem 3.3]{Ky2}, we apply \eqref{3.proofccp.estforsingular} for $\phi=\phi_{i,\epsilon}$ and then let $\epsilon \to 0^+$ to obtain \eqref{3.ccp.nu_mu} with $\mu_i=\mu(\{x_i\}).$ The relation \eqref{3.ccp.form.mu} is also obtained as in \cite[Proof of Theorem 3.3]{Ky2}.
\end{proof}
We close this section by proving Theorem~\ref{III.ccp.infinity}. The proof is quite close to that of \cite[Theorem 4.1]{Saintier}, which uses the ideas in \cite{Chabrowski}.

\begin{proof}[Proof of Theorem~\ref{III.ccp.infinity}]
Let $\phi$ be in $C^\infty(\mathbb{R})$ such that  $\phi(t)\equiv 0$ on $|t|\leq 1,$ $\phi(t)\equiv 1$ on $|t|\geq 2$ and $0\leq \phi\leq 1,$ $|\phi'|\leq 2$ in $\mathbb{R}.$ For each $R>0$, set $\phi_R(x):=\phi(|x|/R)$ for all $x\in\mathbb{R}^N.$ We then decompose
\begin{align}\label{3.ccp.infinity.decompose}
\int_{ \mathbb{R}^N}\big[|\nabla u_n|^{p(x)}&+V(x)|u_n|^{p(x)}\big]\diff x=\int_{ \mathbb{R}^N}\left[|\nabla u_n|^{p(x)}+V(x)|u_n|^{p(x)}\right]\phi_R^{p(x)}\diff x\notag\\
&+\int_{ \mathbb{R}^N}\left[|\nabla u_n|^{p(x)}+V(x)|u_n|^{p(x)}\right](1-\phi_R^{p(x)})\diff x.
\end{align}
For the first term of the right-hand side of \eqref{3.ccp.infinity.decompose} we notice that
\begin{align*}
\int_{\{|x|>2R\}}\big[|\nabla u_n|^{p(x)}&+V(x)|u_n|^{p(x)}\big]\diff x\\
&\leq \int_{ \mathbb{R}^N}\big[|\nabla u_n|^{p(x)}+V(x)|u_n|^{p(x)}\big]\phi_R^{p(x)}\diff x\\
&\quad\quad\quad\leq \int_{\{|x|>R\}}\left[|\nabla u_n|^{p(x)}+V(x)|u_n|^{p(x)}\right]\diff x.
\end{align*}
Thus we obtain
\begin{equation}\label{3.ccp.infinity.mu_infinity}
\mu_\infty=\lim_{R\to\infty}\underset{n\to\infty}{\lim\sup}\int_{\mathbb{R}^N}\left[|\nabla u_n|^{p(x)}+V(x)|u_n|^{p(x)}\right]\phi_R^{p(x)}\diff x.
\end{equation}
For the second term of the right-hand side of \eqref{3.ccp.infinity.decompose}, we notice that $1-\phi_R^{p(x)}$ is a continuous function with compact support in $\mathbb{R}^N$. Hence,
\begin{equation}\label{3.ccp.infinity.phiR}
\lim_{n\to\infty}\int_{\mathbb{R}^N}\left[|\nabla u_n|^{p(x)}+V(x)|u_n|^{p(x)}\right](1-\phi_R^{p(x)})\diff x=\int_{\mathbb{R}^N}(1-\phi_R^{p(x)})\diff\mu.
\end{equation}
Clearly, $\lim_{R\to\infty}\int_{\mathbb{R}^N}\phi_R^{p(x)}\diff\mu=0$ in view of the Lebesgue dominated convergence theorem. By this and \eqref{3.ccp.infinity.phiR}, we deduce
\begin{equation*}
\lim_{R\to\infty}\lim_{n\to \infty }\int_{\mathbb{R}^N}\left[|\nabla u_n|^{p(x)}+V(x)|u_n|^{p(x)}\right](1-\phi_R^{p(x)})\diff x=\mu(\mathbb{R}^N).
\end{equation*}
Using this and \eqref{3.ccp.infinity.mu_infinity}, we obtain \eqref{3.ccp.infinity.mu} by taking the limit superior as $n\to\infty$ and then letting $R\to\infty$ in \eqref{3.ccp.infinity.decompose}. In the same fashion, to obtain \eqref{3.ccp.infinity.nu} we decompose
\begin{equation}\label{3.ccp.infinity.decompose2}
\int_{ \mathbb{R}^N}b(x)|u_n|^{q(x)}\diff x=\int_{ \mathbb{R}^N}b(x)|u_n|^{q(x)}\phi_R^{q(x)}\diff x+\int_{ \mathbb{R}^N}b(x)|u_n|^{q(x)}(1-\phi_R^{q(x)})\diff x.
\end{equation}
Arguing as above, we obtain
\begin{equation}\label{3.ccp.infinity.nu_infinity}
\nu_\infty=\lim_{R\to\infty}\underset{n\to\infty}{\lim\sup}\int_{\mathbb{R}^N}b(x)|u_n|^{q(x)}\phi_R^{q(x)}\diff x,
\end{equation}
and using \eqref{3.ccp.infinity.nu_infinity}, we easily obtain \eqref{3.ccp.infinity.nu} from \eqref{3.ccp.infinity.decompose2}.

Next, we prove \eqref{3.ccp.infinity.nu_mu} when $(\mathcal{E}_\infty)$ is additionally assumed. 
It is easy to see that $\phi_Rv\in X$ for all $R>0$ and all $v\in X$. Thus by \eqref{3.best.const}, we have
\begin{align*}
S|\phi_R u_n|_{L^{q(x)}(b,\mathbb{R}^N)} &\leq \big||\nabla(\phi_R u_n)|\big|_{L^{p(x)}(\mathbb{R}^N)}+|\phi_R u_n|_{L^{p(x)}(V,\mathbb{R}^N)}\\
&\leq \big||\phi_R\nabla u_n|\big|_{L^{p(x)}(\mathbb{R}^N)}+\big||u_n\nabla\phi_R |\big|_{L^{p(x)}(\mathbb{R}^N)}+|\phi_R u_n|_{L^{p(x)}(V,\mathbb{R}^N)}.
\end{align*}
Hence,
\begin{align}\label{3.ccp.infinity.est1}
S|\phi_R u_n|_{L^{q(x)}(b,B_{R}^c)} \leq \big||\phi_R\nabla u_n|\big|_{L^{p(x)}(B_{R}^c)}&+|\phi_R u_n|_{L^{p(x)}(V,B_{R}^c)}\notag\\
&+\big||u_n\nabla\phi_R|\big|_{L^{p(x)}(B_{2R}\setminus \overline{B_{R}})}.
\end{align}
Let $\epsilon$ be arbitrary in $(0,1).$ By $(\mathcal{E}_\infty)$, there exists $R_0=R_0(\epsilon)>0$ such that
\begin{equation}\label{3.ccp.infinity.p,q_infinity}
|p(x)-p_\infty|<\epsilon,\ |q(x)-q_\infty|<\epsilon,\quad \forall |x|>R_0.
\end{equation}
For $R>R_0$ given, let $\{u_{n_k}\}$ be a subsequence of $\{u_n\}$ such that
\begin{equation}\label{3.ccp.infinity.limsup}
\lim_{k\to \infty }\int_{B_R^c}b(x)|u_{n_k}|^{q(x)}\phi_R^{q(x)}\diff x=\underset{n\to\infty}{\lim\sup}\int_{B_R^c}b(x)|u_n|^{q(x)}\phi_R^{q(x)}\diff x.
\end{equation}
Invoking Proposition~\ref{2.prop2} and taking  into account \eqref{3.ccp.infinity.p,q_infinity}, we have
\begin{align*}
&|\phi_R u_{n_k}|_{L^{q(x)}(b,B_{R}^c)}\\
&\geq\min\left\{\left(\int_{B_R^c}b(x)|u_{n_k}|^{q(x)}\phi_R^{q(x)}\diff x\right)^{\frac{1}{q_\infty+\epsilon}},\left(\int_{B_R^c}b(x)|u_{n_k}|^{q(x)}\phi_R^{q(x)}\diff x\right)^{\frac{1}{q_\infty-\epsilon}}\right\}.
\end{align*}
From this, \eqref{3.ccp.infinity.nu_infinity} and \eqref{3.ccp.infinity.limsup} we deduce
\begin{equation}\label{3.ccp.infinity.est2}
\lim_{R\to\infty}\underset{{k}\to\infty}{\lim\sup}|\phi_R u_{n_k}|_{L^{q(x)}(b,B_{R}^c)}\geq \min\biggl\{\nu_\infty^{\frac{1}{q_\infty+\epsilon}} , \nu_\infty^{\frac{1}{q_\infty-\epsilon}} \biggr\}.
\end{equation}
On the other hand, invoking Proposition~\ref{2.prop2} and taking  into account \eqref{3.ccp.infinity.p,q_infinity} again, we have
\begin{align*}
&\big||\phi_R\nabla u_{n_k}|\big|_{L^{p(x)}(B_{R}^c)}\\
&\leq\max\left\{\left(\int_{B_R^c}|\nabla u_{n_k}|^{p(x)}\phi_R^{p(x)}\diff x\right)^{\frac{1}{p_\infty+\epsilon}},\left(\int_{B_R^c}|\nabla u_{n_k}|^{p(x)}\phi_R^{p(x)}\diff x\right)^{\frac{1}{p_\infty-\epsilon}}\right\} \\
&\leq\max\biggl\{\left(\int_{B_R^c}\left[|\nabla u_{n_k}|^{p(x)}+V(x)|u_{n_k}|^{p(x)}\right]\phi_R^{p(x)}\diff x\right)^{\frac{1}{p_\infty+\epsilon}},\\
&{}\hspace*{5cm}\left(\int_{B_R^c}\left[|\nabla u_{n_k}|^{p(x)}+V(x)|u_{n_k}|^{p(x)}\right]\phi_R^{p(x)}\diff x\right)^{\frac{1}{p_\infty-\epsilon}}\biggr\}.
\end{align*}
Similarly we have
\begin{align*}
&\big|\phi_R u_{n_k}\big|_{L^{p(x)}(V,B_{R}^c)}\\
&\leq\max\biggl\{\left(\int_{B_R^c}\left[|\nabla u_{n_k}|^{p(x)}+V(x)|u_{n_k}|^{p(x)}\right]\phi_R^{p(x)}\diff x\right)^{\frac{1}{p_\infty+\epsilon}},\\
&{}\hspace*{5cm}\left(\int_{B_R^c}\left[|\nabla u_{n_k}|^{p(x)}+V(x)|u_{n_k}|^{p(x)}\right]\phi_R^{p(x)}\diff x\right)^{\frac{1}{p_\infty-\epsilon}}\biggr\}.
\end{align*}
Hence,
\begin{align*}
&\big||\phi_R\nabla u_{n_k}|\big|_{L^{p(x)}(B_{R}^c)}+\big|\phi_R u_{n_k}\big|_{L^{p(x)}(V,B_{R}^c)}\\
&\leq2\max\biggl\{\left(\int_{B_R^c}\left[|\nabla u_{n_k}|^{p(x)}+V|u_{n_k}|^{p(x)}\right]\phi_R^{p(x)}\diff x\right)^{\frac{1}{p_\infty+\epsilon}},\\
&{}\hspace*{5cm}\left(\int_{B_R^c}\left[|\nabla u_{n_k}|^{p(x)}+V|u_{n_k}|^{p(x)}\right]\phi_R^{p(x)}\diff x\right)^{\frac{1}{p_\infty-\epsilon}}\biggr\}.
\end{align*}
This and \eqref{3.ccp.infinity.mu_infinity} yield
\begin{equation}\label{3.ccp.infinity.est3}
\lim_{R\to\infty}\underset{{k}\to\infty}{\lim\sup}\left[\big||\phi_R\nabla u_{n_k}|\big|_{L^{p(x)}(B_{R}^c)}+\big|\phi_R u_{n_k}\big|_{L^{p(x)}(V,B_{R}^c)}\right]\leq 2\max\biggl\{\mu_\infty^{\frac{1}{p_\infty+\epsilon}} , \mu_\infty^{\frac{1}{p_\infty-\epsilon}} \biggr\}.
\end{equation}
Finally, we analyze the last term in \eqref{3.ccp.infinity.est1}. 
By Lemma~\ref{III.critical.imb.loc}, we have that $X\hookrightarrow\hookrightarrow L^{p(x)}(B_{2R}\setminus \overline{B_R}).$ Hence,
\begin{equation}\label{3.ccp.infinity.est4}
\underset{{k}\to\infty}{\lim\sup}\big||u_{n_k}\nabla\phi_R|\big|_{L^{p(x)}(B_{2R}\setminus \overline{B_{R}})}=\big||u\nabla\phi_R|\big|_{L^{p(x)}(B_{2R}\setminus \overline{B_{R}})}.
\end{equation}
Note that $u\in L^{p^\ast(x)}(\mathbb{R}^N)$ in view of Lemma~\ref{III.critical.imb}. Applying Proposition~\ref{2.prop1}, we obtain
\begin{align}\label{3.ccp.infinity.est5}
\int_{B_{2R}\setminus \overline{B_{R}}}\big|u\nabla\phi_R\big|^{p(x)}\diff x&\leq 2\big||u|^{p}\big|_{L^{\frac{p^\ast(x)}{p(x)}}(B_{2R}\setminus \overline{B_{R}})}\big||\nabla\phi_R|^{p}\big|_{L^{\frac{N}{p(x)}}(B_{2R}\setminus \overline{B_{R}})}\notag\\
&\leq 2c \big||u|^{p}\big|_{L^{\frac{p^\ast(x)}{p(x)}}(B_{2R}\setminus \overline{B_{R}})},
\end{align}
where $c$ is a constant independent of $R.$ Here we have used the fact that 
$$\big||\nabla\phi_R|^{p}\big|_{L^{\frac{N}{p(x)}}(B_{2R}\setminus \overline{B_{R}})}\leq \left(1+\int_{B_{2R}\setminus \overline{B_{R}}}|\nabla\phi_R|^N\diff x \right)^{\frac{p^+}{N}}\leq \text{constant},\ \forall R>0.$$
 Hence, the estimate \eqref{3.ccp.infinity.est5} infers
$$\lim_{R\to\infty}\int_{B_{2R}\setminus \overline{B_{R}}}\big|u\nabla\phi_R\big|^{p(x)}\diff x=0.$$
 Hence,
 \begin{equation}\label{3.ccp.infinity.est6}
 \lim_{R\to\infty}\big||u\nabla\phi_R|\big|_{L^{p(x)}(B_{2R}\setminus \overline{B_{R}})}=0.
 \end{equation}
Utilizing \eqref{3.ccp.infinity.est2}-\eqref{3.ccp.infinity.est4} and \eqref{3.ccp.infinity.est6} we obtain from \eqref{3.ccp.infinity.est1} that
$$S\min\biggl\{\nu_\infty^{\frac{1}{q_\infty+\epsilon}} , \nu_\infty^{\frac{1}{q_\infty-\epsilon}} \biggr\}\leq 2\max\biggl\{\mu_\infty^{\frac{1}{p_\infty+\epsilon}} , \mu_\infty^{\frac{1}{p_\infty-\epsilon}} \biggr\}.$$
Letting $\epsilon\to 0^+$ in the last inequality, we obtain \eqref{3.ccp.infinity.nu_mu}.
\end{proof}

\section{The existence of solutions}

\subsection{Statements of the existence results}
In this section, we investigate the existence and multiplicity of solutions to the following problem
\begin{eqnarray}\label{Eq}
\begin{cases}
-\operatorname{div}\left(|\nabla u| ^{p(x)-2}\nabla u\right)+V(x)|u|^{p(x)-2}u=\lambda a(x)|u|^{r(x)-2}u+\theta b(x)|u|^{q(x)-2}u  &\text{in}~ \mathbb{R}^N ,\\
u\in X,
\end{cases}
\end{eqnarray} 
with $p,V,q,b$ satisfying $(\mathcal{P})$, $(\mathcal{C}),$ and $(\mathcal{E}_\infty)$ and $r,a$ satisfying
\begin{itemize}
	\item[$(\mathcal{S})$]  $a\in L_+^{\frac{p^\ast(x)}{p^\ast(x)-r(x)}}(\mathbb{R}^N)$ for some $r \in C_+(\mathbb{R}^N)$ such that
	$\underset{x\in\mathbb{R}^N}{\inf}[p^\ast(x)-r(x)]>0.$
\end{itemize}
The condition $(\mathcal{S})$ is crucial for our variational argument due to the following result.
\begin{lemma}\label{IV.compact}
	Under assumptions $(\mathcal{P})$ and $(\mathcal{S})$, we have the following compact continuous imbedding
	\begin{equation*}
	X \hookrightarrow \hookrightarrow L^{r(x) }(a,\mathbb{R}^N ).
	\end{equation*}
\end{lemma}
The proof of this lemma is provided in 	Appendix~\ref{AppendixC}.

\medskip
We say that $ u\in X $ is a (weak) solution of problem \eqref{Eq} if
\begin{align*}
\int_{\mathbb{R}^N}|\nabla u|^{p(x)-2}&\nabla u\cdot \nabla \varphi \diff x+\int_{\mathbb{R}^N} V(x)|u|^{p(x)-2}u\varphi \diff x\\
&=\lambda \int_{\mathbb{R}^N}a(x)|u|^{r(x)-2}u\varphi \diff x+\theta\int_{\mathbb{R}^N}b(x)|u|^{q(x)-2}u\varphi \diff x
\end{align*}
for all $\varphi\in X$. The energy functional $J:X \to\mathbb{R}$ associated with problem \eqref{Eq} is defined as
\begin{align*}
J(u) =\int_{\mathbb{R}^N}\frac{1}{p(x) }| \nabla u| ^{p(x) }\diff x&+\int_{\mathbb{R}^N}\frac{V(x)}{p(x) }|u| ^{p(x) }\diff x\\
&-\lambda\int_{\mathbb{R}^N}\frac{a(x)}{r(x)}|u|^{r(x)}\diff x-\theta\int_{\mathbb{R}^N}\frac{b(x)}{q(x)}|u|^{q(x)}\diff x.
\end{align*}
Note that the definitions above are well-defined thanks to Lemmas~\ref{III.critical.imb} and \ref{IV.compact}. We shall investigate the existence of solutions to problem \eqref{Eq} in two cases: the case $r^->p^+$ (the $p(\cdot)$-superlinear problem) and the case $r^+\leq p^-$ (the concave-convex type problem). We first state our main existence result for the $p(\cdot)$-superlinear problem.

\begin{theorem} \label{IV.main1}
	Assume that $(\mathcal{P}),$ $(\mathcal{C}),$ $(\mathcal{E}_\infty),$  and $(\mathcal{S})$ hold such that $p^+<\min\{r^-,q^-\}.$
	\begin{itemize}
		\item[(i)]   For each given $\theta >0$, there exists $\lambda_0=\lambda_0(\theta)>0$ such that for any $\lambda>\lambda_0 ,$  problem \eqref{Eq} has a nontrivial nonnegative solution $\overline{u}$ with $J(\overline{u})>0$.
		\item[(ii)]   For each given $\lambda>0$, there exists $\theta_0=\theta_0(\lambda)>0$ such that for any $\theta \in (0,\theta_0 ),$  problem \eqref{Eq} has a nontrivial nonnegative solution  $\overline{u}$ with $J(\overline{u})>0$.
	\end{itemize}
\end{theorem}

For the concave-convex type problem, we have the following results.


\begin{theorem} \label{IV.main2}
	Assume that $(\mathcal{P}),$ $(\mathcal{C}),$ $(\mathcal{E}_\infty),$ $(\mathcal{P}),$ and $(\mathcal{S})$ hold such that $r^+\leq p^-\leq p^+<q^-$ and  $a^{\frac{q}{q-r}}b^{-\frac{r}{q-r}}\in L^1(\mathbb{R}^N).$
	\begin{itemize}
		\item[(i)]   If $r^+<p^-\leq p^+<q^-$ or $r^+=p^-\leq p^+<q^-$ with $\{x\in \mathbb{R}^N: r(x)<p(x)\}\ne \emptyset$, then for each given $\theta >0$, there exists $\lambda_{\ast}=\lambda_{\ast}(\theta)>0$ such that for any $\lambda\in(0,\lambda_{\ast}) ,$  problem \eqref{Eq} has a nontrivial nonnegative solution with negative energy.
		\item[(ii)]   If $r^+<p^-\leq p^+<q^-$ and $(\frac{q}{p})^+<(\frac{q}{r})^-$, then for each given $\lambda>0$, there exists $\theta_{\ast}=\theta_{\ast}(\lambda)>0$ such that for any $\theta \in (0,\theta_{\ast} ),$  problem \eqref{Eq} has a nontrivial nonnegative solution with negative energy.
	\end{itemize}
\end{theorem}
\begin{theorem} \label{IV.main3}
	Assume that $(\mathcal{P}),$ $(\mathcal{C}),$ $(\mathcal{E}_\infty),$  and $(\mathcal{S})$ hold such that $r^+<p^-\leq p^+<q^-$ and $a^{\frac{q}{q-r}}b^{-\frac{r}{q-r}}\in L^1(\mathbb{R}^N).$
	\begin{itemize}
		\item[(i)]   For each given $\theta >0$, there exists $\lambda^{\ast}=\lambda^{\ast}(\theta)>0$ such that for any $\lambda\in(0,\lambda^{\ast}) ,$  problem \eqref{Eq} has infinitely many solutions with negative energy.
		\item[(ii)]  Assume also that if $(\frac{q}{p})^+<(\frac{q}{r})^-,$ then for each given $\lambda>0,$ there exists $\theta^{\ast}=\theta^{\ast}(\lambda)>0$ such that for any $\theta \in (0,\theta^{\ast} ),$ problem \eqref{Eq} has infinitely many solutions with negative energy.
	\end{itemize}
\end{theorem}
\begin{remark}\label{V.cri-prob.ex}
	(i) 
	When $\Omega=\emptyset,$ $q(x)=p^\ast(x)$ for all $x\in \mathbb{R}^N,$ and $0<b_1 \le b(x) \le b_2$ a.e. in $\mathbb{R}^N$ for some positive constants $b_1,b_2$, the condition  $a^{\frac{q}{q-r}}b^{-\frac{r}{q-r}}\in L^1(\mathbb{R}^N)$ holds automatically if assumption $(\mathcal{S})$ holds.
	
	(ii) When $V\in L^\infty(\mathbb{R}^N)$ and $p\in C^1(\mathbb{R}^N)$, all nontrivial nonnegative solutions above are positive due to the strong maximum principle for $p(x)$-Laplacian (see \cite [Proposition 3.1]{Fan4}).
	
\end{remark}
\subsection{Proofs of existence results}
For  $u\in X,$ we denote $u^+=\max \{u,0\}, u^-=\max \{-u,0\}$.
To find nonnegative solutions, instead of working with $J$, we will argue with the functional $J_1:X \to\mathbb{R}$ defined as
\begin{align*}
J_1(u) =\int_{\mathbb{R}^N }\frac{1}{p(x) }| \nabla u| ^{p(x) }\diff x&+\int_{\mathbb{R}^N}\frac{V(x)}{p(x) }|u| ^{p(x) }\diff x\\
&-\lambda\int_{\mathbb{R}^N}\frac{a(x)}{r(x)}(u^+)^{r(x)}\diff x-\theta\int_{\mathbb{R}^N}\frac{b(x)}{q(x)}(u^+)^{q(x)}\diff x.
\end{align*}
It is not difficult to see that $J,J_1$ are of $C^{1}(X,\mathbb{R})$ thanks to Lemmas~\ref{III.critical.imb} and \ref{IV.compact}. Clearly, a critical point of $J$ (resp. $J_1$) is a solution (resp. a nonnegative solution) of problem \eqref{Eq}. To prove our main results, we employ techniques of the calculus of variations to determine the critical points of $J_1$ or $J$. The following lemma, the so called $(S_+)$-property, is useful for showing the $\textup{(PS)}$ condition. Since the proof can be obtained via a standard argument (see e.g., \cite[Proof of Lemma 3.2]{Sim-Kim}), we omit it.
\begin{lemma} \label{IV.S+}
	If $u_{n}\rightharpoonup u$ in  $X$ and
	$$\limsup_{n\to  \infty } \int_{\mathbb{R}^N}\left[|\nabla u_n|^{p(x)-2}\nabla u_n\cdot(\nabla u_n-\nabla u)+V(x)|u_n|^{p(x)-2}u_n(u_n-u)\right]\diff x\leq 0,$$
	then  $u_{n}\to  u$ in $X$.
\end{lemma}

\subsubsection{Proof of Theorem~\ref{IV.main1}}
To show the existence of solutions to the $p(\cdot)$-superlinear problem, we employ the Mountain Pass Theorem to determine the critical points of $J_1.$ The next lemma shows the geometry of $J_1$. We skip its proof because it is almost identical to \cite[Proof of Lemma 4.2]{Ky2} by employing Lemmas~\ref{III.critical.imb} and \ref{IV.compact}.

\begin{lemma} \label{IV.geometry1}
	Assume that $(\mathcal{P}),$ $(\mathcal{C})$,  and $(\mathcal{S})$ hold such that $p^+<\min\{r^-,q^-\}.$ For each given  $\lambda>0,\theta>0,$ there  exist $0<\delta<1$ and $\rho>0$ such that
	$J_1(u)\geq \rho \ \ \text{if}\  \ \|u\|=\delta.$
\end{lemma}
	
The information about the compactness condition for $J_1$ is given in the next lemma.
\begin{lemma} \label{IV.PS1}
	Assume that $(\mathcal{P}),$ $(\mathcal{C}),$ $(\mathcal{E}_\infty)$,  and $(\mathcal{S})$ hold such that $p^+<\min\{r^-,q^-\}.$ For each given $\lambda>0,\theta>0,$ $J_1$ satisfies the $(PS)_c$ condition for all
	\begin{equation}\label{4.condition1.c}
	c<(\frac{1}{p^+}-\frac{1}{q^-})\frac{1}{2^{(qh)^+}}\min\{S^{(qh)^+},S^{(qh)^-}\}\min\{\theta^{-h^+},\theta^{-h^-}\},
	\end{equation}
	where $h(x):=\frac{p(x)}{q(x)-p(x)}$ for all $x\in\mathbb{R}^N,$  and $S$ is defined as in \eqref{3.best.const}.
\end{lemma}
\begin{proof}
	Let $c$ satisfy \eqref{4.condition1.c} and let $\{u_n\}$ be a $\textup{(PS)}_c$ sequence for $J_1$. As in \cite[Proof of Lemma 4.3]{Ky2}, we obtain that $\{u_n\}$ is bounded in $X$ and $\{u_n^+\}$ is also a $\textup{(PS)}_c$ sequence for $J_1$. Set $v_n:=u_n^+\geq 0$ ($n=1,2,\cdots$). Then, from Theorem~\ref{III.ccp}, Theorem~\ref{III.ccp.infinity} and Lemma~\ref{IV.compact}, up to a subsequence, we have
	\begin{gather}
	v_n(x) \to v(x)\geq 0  \quad \text{for a.e.} \ \ x\in\mathbb{R}^N,\label{4.a.e1}\\
	v_n \rightharpoonup v  \quad \text{in} \  X,\label{4.weak1}\\
	|\nabla v_n|^{p(x)}+ V(x)|v_n|^{p(x)}\overset{\ast }{\rightharpoonup }\mu \geq |\nabla v|^{p(x)} +V(x)|v|^{p(x)}+ \sum_{i\in I} \mu_i \delta_{x_i} \ \text{in}\  \mathcal{M}(\mathbb{R}^N),\label{4.mu1}\\
	b|v_n|^{q(x)}\overset{\ast }{\rightharpoonup }\nu=b|v|^{q(x)} + \sum_{i\in I}\nu_i\delta_{x_i} \ \text{in}\ \mathcal{M}(\mathbb{R}^N),\label{4.nu1}\\
	S \nu_i^{1/q(x_i)} \leq \mu_i^{1/p(x_i)}, \quad \forall i\in I,\label{4.mu-nu1}
		\end{gather}
	and
		\begin{gather}
	\underset{n\to\infty}{\lim\sup}\int_{\mathbb{R}^N}\left[|\nabla v_n|^{p(x)}+V(x)|v_n|^{p(x)}\right]\diff x=\mu(\mathbb{R}^N)+\mu_\infty,\label{4.mu_infinity1}\\
	\underset{n\to\infty}{\lim\sup}\int_{\mathbb{R}^N}b(x)|v_n|^{q(x)}\diff x=\nu(\mathbb{R}^N)+\nu_\infty,\label{4.nu_infinity1}\\
	\frac{1}{2}S \nu_\infty^{1/q_\infty} \leq \mu_\infty^{1/p_\infty}.\label{4.mu_inf-nu_inf}
	\end{gather}

We claim that $I=\emptyset$ and $\nu_\infty=0.$ To this aim, we first estimate
	\begin{align}\label{4.PS1.est.c}
	c&=\lim_{n\to\infty}[J_1(v_n)-\frac{1}{p^+}\langle J'_1(v_n) ,v_n\rangle]\notag\\
	&\geq \left(\frac{1}{p^+}-\frac{1}{q^-}\right)\theta\underset{n\to\infty}{\lim\sup}\int_{\mathbb{R}^N}b(x)v_n^{q(x)}\diff x=\left(\frac{1}{p^+}-\frac{1}{q^-}\right)\theta\left[\nu(\mathbb{R}^N)+\nu_\infty\right].
	\end{align}
	We now prove that $I=\emptyset.$ Suppose by contradiction that there exists $i\in I$. Let $\epsilon>0$ define $\phi_{i,\epsilon}$ as in the proof of Theorem~\ref{III.ccp}. We have
	\begin{align*}
	\int_{\mathbb{R}^N}|\nabla v_n|^{p(x)-2}&v_n\nabla v_n\cdot\nabla \phi_{i,\epsilon}\diff x=\langle J'_1(v_n) ,\phi_{i,\epsilon}v_n \rangle+\lambda\int_{\mathbb{R}^N}\phi_{i,\epsilon}a(x)v_n^{r(x)}\diff x\\
	&+\theta\int_{\mathbb{R}^N}\phi_{i,\epsilon}b(x)v_n^{q(x)}\diff x-\int_{\mathbb{R}^N}\phi_{i,\epsilon}[|\nabla v_n|^{p(x)}+V(x)|v_n|^{p(x)}]\diff x.
	\end{align*}
Arguing as in \cite[Proof of Lemma 4.3]{Ky2} and taking into account the boundedness of $\{v_n\}$ in $X$ as well as the fact that $v_n\to v$ in both $L^{r(x)}(a,\mathbb{R}^N)$ and $L^{p(x)}(B_\epsilon(x_i))$ (see Lemmas~\ref{III.critical.imb.loc} and \ref{IV.compact}), we deduce from the last equality (after taking the limit as $n\to\infty$ and then $\epsilon\to 0^+$) that
	$\mu_i=\theta\nu_i.$ Thus, \eqref{4.mu-nu1} gives
	$$\mu_i\geq S^{\frac{q(x_i)p(x_i)}{q(x_i)-p(x_i)}}\theta^{-\frac{p(x_i)}{q(x_i)-p(x_i)}}.$$
	Hence,
	\begin{equation*}
	\theta\nu_i=\mu_i\geq \min\{S^{(qh)^+},S^{(qh)^-}\}\min\{\theta^{-h^+},\theta^{-h^-}\}.
	\end{equation*}
	From this and \eqref{4.PS1.est.c}, we get
	$$c\geq \left(\frac{1}{p^+}-\frac{1}{q^-}\right)\theta\nu_i\geq  \left(\frac{1}{p^+}-\frac{1}{q^-}\right)\min\{S^{(qh)^+},S^{(qh)^-}\}\min\{\theta^{-h^+},\theta^{-h^-}\}.$$
	This contradicts \eqref{4.condition1.c}, and hence; $I=\emptyset.$ We next show that $\nu_\infty=0.$ Let $\phi_R$ be the same as in the proof of Theorem~\ref{III.ccp.infinity}. By using $\phi_R$ in place of $\phi_R^{p(x)}$ and $\phi_R^{q(x)}$ as well as  using $v_n$ in place of $u_n$ in \eqref{3.ccp.infinity.decompose} and \eqref{3.ccp.infinity.decompose2}, we easily obtain
	\begin{equation}\label{3.ccp.infinity.mu_infinity2}
	\mu_\infty=\lim_{R\to\infty}\underset{n\to\infty}{\lim\sup}\int_{\mathbb{R}^N}\left[|\nabla v_n|^{p(x)}+V(x)|v_n|^{p(x)}\right]\phi_R\diff x,
	\end{equation}
	and
	\begin{equation}\label{3.ccp.infinity.nu_infinity2}
	\nu_\infty=\lim_{R\to\infty}\underset{n\to\infty}{\lim\sup}\int_{\mathbb{R}^N}b(x)|v_n|^{q(x)}\phi_R\diff x.
	\end{equation}
	 Since $\{\phi_Rv_n\}$ is bounded in $X$, we have
	\begin{align}\label{4.PS1.nu_infinity}
	o_n(1)&=\langle J'_1(v_n) ,\phi_Rv_n \rangle\notag\\
	&=\int_{\mathbb{R}^N}\left[|\nabla v_n|^{p(x)}+V(x)|v_n|^{p(x)}\right]\phi_R\diff x+\int_{\mathbb{R}^N}|\nabla v_n|^{p(x)-2}v_n\nabla v_n\cdot\nabla \phi_R\diff x\notag\\
	&\quad\quad-\theta\int_{\mathbb{R}^N}b(x)v_n^{q(x)}\phi_R\diff x-\lambda\int_{\mathbb{R}^N}a(x)v_n^{r(x)}\phi_R\diff x.
	\end{align}
	Noticing that $v_n\to v$ in $L^{r(x)}(a,\mathbb{R}^N),$ we have
	\begin{equation}\label{4.PS1.ets.a}
	\lim_{n\to \infty }\int_{\mathbb{R}^N}a(x)v_n^{r(x)}\phi_R\diff x=\int_{\mathbb{R}^N}a(x)v^{r(x)}\phi_R\diff x.
	\end{equation}
	Next, we shall prove that
	\begin{equation}\label{4.mu-nu_infinity}
	\lim_{R\to\infty}\underset{n\to\infty}{\lim\sup}\bigg|\int_{\mathbb{R}^N}|\nabla v_n|^{p(x)-2}v_n\nabla v_n\cdot\nabla \phi_R\diff x\bigg|=0.
	\end{equation}
	Indeed, using Young's inequality, for each $\delta>0$, we have
	$$\bigg|\int_{\mathbb{R}^N}|\nabla v_n|^{p(x)-2}v_n\nabla v_n\cdot\nabla \phi_Rdx\bigg|\leq \delta \int_{\mathbb{R}^N}|\nabla v_n|^{p(x)}\diff x+C(\delta)\int_{\mathbb{R}^N}|v_n\nabla \phi_R|^{p(x)}\diff x.$$
	Taking into account the boundedness of $\{v_n\}$ in $X$ and $v_n\to v$ in $L^{p(x)}(B_{2R}\setminus \overline{B_{R}})$ (see Lemma~\ref{III.critical.imb.loc}), we obtain from the last inequality that
	\begin{equation}\label{4.mu-nu.infinity.est1}
\underset{n\to\infty}{\lim\sup}\bigg|\int_{\mathbb{R}^N}|\nabla v_n|^{p(x)-2}v_n\nabla v_n\cdot\nabla \phi_R\diff x\bigg|
\leq C\delta +C(\delta)\int_{\mathbb{R}^N}|v\nabla \phi_R|^{p(x)}\diff x,
	\end{equation}
where $C$ is a positive constant independent of $\delta.$ Using a similar argument to that obtained \eqref{3.ccp.infinity.est6}, we have
	$$\int_{\mathbb{R}^N}|v\nabla \phi_R|^{p(x)}dx\to 0\ \ \text{as}\ \ R \to \infty.$$
From this and \eqref{4.mu-nu.infinity.est1}, we infer
	$$\lim_{R\to\infty}\underset{n\to\infty}{\lim\sup}\bigg|\int_{\mathbb{R}^N}|\nabla v_n|^{p(x)-2}v_n\nabla v_n\cdot\nabla \phi_R\diff x\bigg|\leq C\delta .$$
	Since $\delta>0$ was arbitrarily chosen, we hence obtain \eqref{4.mu-nu_infinity}. 
Combining \eqref{4.PS1.nu_infinity} with \eqref{4.PS1.ets.a} and \eqref{4.mu-nu_infinity} we deduce
	\begin{align*}
	\underset{n\to\infty}{\lim\sup}\int_{\mathbb{R}^N}&\left[|\nabla v_n|^{p(x)}+V(x)|v_n|^{p(x)}\right]\phi_R\diff x\leq \theta\underset{n\to\infty}{\lim\sup}\int_{\mathbb{R}^N}b(x)v_n^{q(x)}\phi_R\diff x\notag\\
	&+ \lambda \int_{\mathbb{R}^N}a(x)v^{r(x)}\phi_R\diff x+\underset{n\to\infty}{\lim\sup}\bigg|\int_{\mathbb{R}^N}|\nabla v_n|^{p(x)-2}v_n\nabla v_n\cdot\nabla \phi_R\diff x\bigg|.
	\end{align*}
Letting $R\to\infty$ in the last inequality and noting \eqref{3.ccp.infinity.mu_infinity2} and \eqref{3.ccp.infinity.nu_infinity2}, we get
	$$\mu_\infty\leq\theta\nu_\infty.$$
Similarly we obtain from \eqref{4.PS1.nu_infinity}-\eqref{4.mu-nu_infinity} that
	\begin{align*}
\theta\underset{n\to\infty}{\lim\sup}\int_{\mathbb{R}^N}&b(x)v_n^{q(x)}\phi_R\diff x\leq \underset{n\to\infty}{\lim\sup}\int_{\mathbb{R}^N}\left[|\nabla v_n|^{p(x)}+V(x)|v_n|^{p(x)}\right]\phi_R\diff x\notag\\
&+ \lambda \int_{\mathbb{R}^N}a(x)v^{r(x)}\phi_R\diff x+\underset{n\to\infty}{\lim\sup}\bigg|\int_{\mathbb{R}^N}|\nabla v_n|^{p(x)-2}v_n\nabla v_n\cdot\nabla \phi_R\diff x\bigg|.
\end{align*}	
From this, \eqref{3.ccp.infinity.mu_infinity2} and \eqref{3.ccp.infinity.nu_infinity2} we get
$$\theta\nu_\infty\leq \mu_\infty.$$	
Then, we obtain 	$\mu_\infty=\theta\nu_\infty.$ Thus, if $\nu_\infty>0$, the relation \eqref{4.mu_inf-nu_inf} gives
\begin{equation}\label{4.PS1.est.mu_infinity}
\theta\nu_\infty=\mu_\infty\geq (\frac{1}{2}S)^{\frac{q_\infty p_\infty}{q_\infty -p_\infty}}\theta^{-\frac{p_\infty}{q_\infty -p_\infty}}.
\end{equation}
Note that $\frac{q_\infty p_\infty}{q_\infty -p_\infty}=\lim_{|x|\to\infty} q(x)h(x)$ and $\frac{p_\infty}{q_\infty -p_\infty}=\lim_{|x|\to\infty} h(x)$. So we have
$$(qh)^-\leq \frac{q_\infty p_\infty}{q_\infty -p_\infty}\leq (qh)^+,\quad h^-\leq \frac{p_\infty}{q_\infty -p_\infty}\leq h^+.$$
Using this, we infer from \eqref{4.PS1.est.mu_infinity} that
$$\theta\nu_\infty\geq \min\big\{(\frac{1}{2}S)^{(qh)^+},(\frac{1}{2}S)^{(qh)^-}\big\}\min\{\theta^{-h^+},\theta^{-h^-}\}.$$	
From this and \eqref{4.PS1.est.c}, we obtain	
$$c\geq\left(\frac{1}{p^+}-\frac{1}{q^-}\right)\theta\nu_\infty\geq \left(\frac{1}{p^+}-\frac{1}{q^-}\right)\frac{1}{2^{(qh)^+}}\min\{S^{(qh)^+},S^{(qh)^-}\}\min\{\theta^{-h^+},\theta^{-h^-}\}.$$		
This contradicts \eqref{4.condition1.c}, and hence; $\nu_\infty=0.$	

By this and the fact that $I=\emptyset$, we deduce from \eqref{4.nu1} and \eqref{4.nu_infinity1} that
$$\underset{n\to\infty}{\lim\sup}\int_{\mathbb{R}^N}b(x)v_n^{q(x)}\diff x=\int_{\mathbb{R}^N}b(x)v^{q(x)}\diff x.$$
By \eqref{4.a.e1} we obtain
$$\int_{\mathbb{R}^N}b(x)v^{q(x)}\diff x\leq \underset{n\to\infty}{\lim\inf}\int_{\mathbb{R}^N}b(x)v_n^{q(x)}\diff x$$
in view of Fatou's lemma. Thus,
$$\lim_{n\to\infty}\int_{\mathbb{R}^N}b(x)v_n^{q(x)}\diff x=\int_{\mathbb{R}^N}b(x)v^{q(x)}\diff x.$$
From this, \eqref{4.a.e1} and Lemma~\ref{III.brezis-lieb}, we obtain
	$$\int_{\mathbb{R}^N}b(x)|v_n-v|^{q(x)}\diff x\to 0,\ \ \text{i.e.,}\ \ v_n\to v\ \ \text{in}\ \ L^{q(x)}(b,\mathbb{R}^N).$$
This fact and Proposition~\ref{2.prop1} yield $\int_{\mathbb{R}^N}b(x)v_n^{q(x)-1}(v_n-v)\diff x\to 0$. On the other hand, $\int_{\mathbb{R}^N}a(x)v_n^{r(x)-1}(v_n-v)\diff x\to 0$ due to Lemma~\ref{IV.compact} and \eqref{4.weak1}. We therefore deduce
	\begin{align*}
	\int_{\mathbb{R}^N}&|\nabla v_n|^{p(x)-2}\nabla v_n\cdot\nabla(v_n-v)\diff x+\int_{\mathbb{R}^N}V(x)|v_n|^{p(x)-2}v_n(v_n-v)\diff x\\
	&=\langle J'_1(v_n) ,v_n-v\rangle+\lambda\int_{\mathbb{R}^N}a(x)v_n^{r(x)-1}(v_n-v)\diff x +\theta\int_{\mathbb{R}^N}b(x)v_n^{q(x)-1}(v_n-v)\diff x\\
	&\quad\to\quad 0.
	\end{align*}
Then, by Lemma~\ref{IV.S+}, we obtain $v_n\to v$ in $X$, and hence $u_n=v_n-u_n^-\to v$ in $X$. This completes the proof.
\end{proof}
We are now in a position to give a proof of Theorem~\ref{IV.main1}.
The proof is similar to that of \cite[Theorem 4.1]{Ky2} by invoking Lemmas~\ref{IV.S+}-\ref{IV.PS1} above and we omit it.

\subsubsection{Proofs of Theorems \ref{IV.main2} and \ref{IV.main3}}

In the rest of this section, we provide proofs of Theorems \ref{IV.main2} and \ref{IV.main3} using the Ekeland variational principle and genus theory. To do this, we need several auxiliary results. The next two lemmas provide some geometric results for the energy functional of the concave-convex type problem, and their proofs can be found in \cite[Proofs of Lemmas 3.5 and 3.6]{Ky1}.

\begin{lemma} \label{IV.geometry2.1}
	Assume that $(\mathcal{P}),$ $(\mathcal{C}),$ and $(\mathcal{S})$ hold.
	\begin{itemize}
		\item[(i)] If $p^+ < q^-$, then for each given $\theta>0$, there exists $\lambda_1=\lambda_1(\theta)>0$ such that, for any $\lambda \in(0,\lambda_1)$, there exist $\delta, \rho>0$ such that $J_1(u)\geq \rho$ if $\|u\|=\delta$.
		\item[(ii)] If $r^+<p^-$, then for each given $\lambda>0$, there exists $\theta_1=\theta_1(\lambda)>0$ such that, for any $\theta\in(0,\theta_1)$, there exist $\overline{\delta}, \overline{\rho}>0$ such that $J_1(u)\geq \overline{\rho}$ if $\|u\|=\overline{\delta}$.
	\end{itemize}
\end{lemma}

\begin{lemma} \label{IV.geometry2.2}
	Assume that $(\mathcal{P}),$ $(\mathcal{C}),$ and $(\mathcal{S})$ hold. If $\{x\in \mathbb{R}^N: r(x)<p(x)\}\ne \emptyset$, then for any $\lambda >0,\theta>0$, there exists $\phi \in X, \phi \geq 0$ such that $J_1(t\phi)<0$ for all $t>0$ small.
\end{lemma}

We have the following local compactness condition, which is essential for our argument to obtain the existence of solutions.
\begin{lemma} \label{IV.PS2}
	Assume that $(\mathcal{P}),$ $(\mathcal{C}),$ $(\mathcal{E}_\infty),$ and $(\mathcal{S})$ hold such that $r^+\leq p^-\leq p^+<q^-$ and  $a^{\frac{q}{q-r}}b^{-\frac{r}{q-r}}\in L^1(\mathbb{R}^N).$ If $r^+<p^-\leq p^+<q^-$, then  for each $\lambda>0,\ \theta>0$, there exists a positive constant $K=K(p,q,r,a,b)$ such that $J_1$ satisfies the $(PS)_c$ condition for all
	\begin{align*}
	c<(\frac{1}{p^+}-\frac{1}{q^-})&\frac{1}{2^{(qh)^+}}\min\{S^{(qh)^+},S^{(hq)^-}\}\min\{\theta^{-h^+},\theta^{-h^-}\}\\
	&-K\max\{\theta^{-\frac{1}{l^+-1}},\theta^{-\frac{1}{l^--1}}\}\max\{\lambda^{\frac{l^+}{l^+-1}},\lambda^{\frac{l^-}{l^--1}}\},
	\end{align*}
	where $S$ is defined as in \eqref{3.best.const} and $h(x):=\frac{p(x)}{q(x)-p(x)},\ l(x):=\frac{q(x)}{r(x)}$ for all $x\in\mathbb{R}^N$. If $r^+= p^-\leq p^+<q^-$ and $r^-<p^+,$ the conclusion remains valid for $0<\lambda<\lambda_2:=\frac{\frac{1}{p^+}-\frac{1}{q^-}}{(\frac{1}{r^-}-\frac{1}{q^-})C_r^{r^+}},$ where $C_r$ is an imbedding constant for $X\hookrightarrow L^{r(x)}(a,\mathbb{R}^N)$ and $\theta>0.$
\end{lemma}
\begin{proof}
	Let $\lambda>0,\ \theta>0$ and let $\{u_n\}$ be a $\textup{(PS)}_c$ sequence for $J_1$. As in \cite[Proof of Lemma 4.8]{Ky2}, we easily deduce the boundedness of $\{u_n\}$ in both cases $r^+<p^-\leq p^+<q^-$ with any $\lambda>0,\ \theta>0$ and $r^+= p^-\leq p^+<q^-$ with $0<\lambda<\lambda_2,\ \theta>0$. Moreover, we have that $u_n^-\to 0$ in $X$ and up to a subseqence of $v_n:=u_n^+\ (n=1,2,\cdots),$
	\begin{gather*}
	v_n(x) \to v(x)\geq 0  \quad \text{for a.e.} \ \ x\in\mathbb{R}^N,\label{4.a.e2}\\
	v_n \rightharpoonup v  \quad \text{in} \  X,\label{4.weak2}\\
	|\nabla v_n|^{p(x)} +V|v_n|^{p(x)}\overset{\ast }{\rightharpoonup }\mu \geq |\nabla v|^{p(x)}+V|v|^{p(x)} + \sum_{i\in I} \mu_i \delta_{x_i} \ \text{in}\  \mathcal{M}(\mathbb{R}^N),\label{4.mu2}\\
	b|v_n|^{q(x)}\overset{\ast }{\rightharpoonup }\nu=b|v|^{q(x)} + \sum_{i\in I}\nu_i\delta_{x_i} \ \text{in}\ \mathcal{M}(\mathbb{R}^N),\label{4.nu2}\\
	S \nu_i^{1/q(x_i)} \leq \mu_i^{1/p(x_i)}, \ \forall i\in I,\label{4.mu-nu2}
	\end{gather*}
	and
	\begin{gather*}
	\underset{n\to\infty}{\lim\sup}\int_{\mathbb{R}^N}\left[|\nabla v_n|^{p(x)}+V(x)|v_n|^{p(x)}\right]\diff x=\mu(\mathbb{R}^N)+\mu_\infty,\label{4.mu_infinity2}\\
	\underset{n\to\infty}{\lim\sup}\int_{\mathbb{R}^N}b(x)|v_n|^{q(x)}\diff x=\nu(\mathbb{R}^N)+\nu_\infty,\label{4.nu_infinity2}\\
	\frac{1}{2}S \nu_\infty^{1/q_\infty} \leq \mu_\infty^{1/p_\infty}.\label{4.mu_inf-nu_inf2}
	\end{gather*}
	As before, we want to show $I = \emptyset$ and $\nu_\infty=0.$ Suppose by contradiction that $I \ne \emptyset$ or $\nu_\infty>0.$ We have
	\begin{align}\label{4.PS2.est.c}
	c&=\lim_{n\to\infty}[J_1(v_n)-\frac{1}{p^+}\langle J'_1(v_n) ,v_n\rangle]\notag\\
	&\geq \underset{n\to\infty}{\lim\sup}\left[\theta\left(\frac{1}{p^+}-\frac{1}{q^-}\right)\int_{\mathbb{R}^N}b(x)v_n^{q(x)}\diff x-\lambda \left(\frac{1}{r^-}-\frac{1}{p^+}\right)\int_{\mathbb{R}^N}a(x)v_n^{r(x)}\diff x\right]\notag\\
	&=\theta\left(\frac{1}{p^+}-\frac{1}{q^-}\right)\left[\nu(\mathbb{R}^N)+\nu_\infty\right]-\lambda \left(\frac{1}{r^-}-\frac{1}{p^+}\right)\int_{\mathbb{R}^N}a(x)v^{r(x)}\diff x\notag\\
	&=\theta\left(\frac{1}{p^+}-\frac{1}{q^-}\right)\left[\int_{\mathbb{R}^N}b(x)v^{q(x)}\diff x+{\sum_{i\in I\cup\{\infty\}}\nu_i}\right]-\lambda \left(\frac{1}{r^-}-\frac{1}{p^+}\right)\int_{\mathbb{R}^N}a(x)v^{r(x)}\diff x.
	\end{align}
	If there exists $i\in I$, then as in the proof of Lemma~\ref{IV.PS1}, we deduce
$$	\theta \nu_i=\mu_i\geq \min\{S^{(qh)^+},S^{(qh)^-}\}\min\{\theta^{-h^+},\theta^{-h^-}\}.$$
	If $\nu_\infty>0,$ then as in the proof of Lemma~\ref{IV.PS1} we also deduce
	\begin{align*}
\theta \nu_\infty=\mu_\infty&\geq \min\big\{(\frac{1}{2}S)^{(qh)^+},(\frac{1}{2}S)^{(qh)^-}\big\}\min\{\theta^{-h^+},\theta^{-h^-}\}\\
&\geq \frac{1}{2^{(qh)^+}}\min\{S^{(qh)^+},S^{(qh)^-}\}\min\{\theta^{-h^+},\theta^{-h^-}\}=:k(\theta).
	\end{align*}
From these facts together with \eqref{4.PS2.est.c}, we obtain
	$$c\geq \theta\left(\frac{1}{p^+}-\frac{1}{q^-}\right)\int_{\mathbb{R}^N}b(x)v^{q(x)}\diff x+ \left(\frac{1}{p^+}-\frac{1}{q^-}\right)k(\theta)-\lambda \left(\frac{1}{r^-}-\frac{1}{p^+}\right)\int_{\mathbb{R}^N}a(x)v^{r(x)}\diff x.$$
	The rest of the proof is similar to the proof of \cite[Lemma 4.8]{Ky2}.
\end{proof}
\begin{remark}\label{IV.J}
	It is worth noting that Lemmas~\ref{IV.geometry2.1}-\ref{IV.PS2} remain valid for $J$ instead of $J_1$ with similar proofs (to obtain the conclusion of Lemma~\ref{IV.PS2} for $J$ in place of $J_1,$ we directly argue with $\{u_n\}$ in place of $\{v_n\}$).
\end{remark}

The proofs of Theorems \ref{IV.main2} and \ref{IV.main3} are almost identical with that of \cite[Theorems 4.4 and 4.5]{Ky2} if we replace their lemmas by   Lemmas ~\ref{IV.geometry2.1}-\ref{IV.PS2} and notice Remark~\ref{IV.J}.


\appendix
\section{Proofs of Lemmas~\ref{III.critical.imb.loc} and \ref{III.critical.imb}} \label{AppendixA}
\begin{proof}[Proof of Lemma~\ref{III.critical.imb.loc}]
	The conclusion is obvious since $W_V^{1,p(x)}\left(B_R\right)=W^{1,p(x)}\left(B_R\right)$ for any $R>0$ due to the assumption of the lemma.
\end{proof}
\begin{proof}[Proof of Lemma~\ref{III.critical.imb}]
	
	We shall prove that there exists a positive constant $C$ such that
\begin{equation}\label{1imbedding.ineq}
|u|_{L^{q(x) }\left(b,\mathbb{R}^N \right)}\leq C\|u\|,\quad \forall u\in   X.
\end{equation}	
We only prove \eqref{1imbedding.ineq} for the case $\Omega\ne\emptyset$ since the proof for the case $\Omega=\emptyset$ is similar but also simpler.  By definition of $X$, it suffices to prove \eqref{1imbedding.ineq} for $u\in C_c^\infty(\mathbb{R}^N).$ Let $u\in C_c^\infty(\mathbb{R}^N)$ be arbitrary. In the rest of the proof, $C_i$ ($i\in\mathbb{N}$) is a positive constant independent of $u$. Utilizing assumptions $(\mathcal{C})$, $(\mathcal{P})$, and Propositions~\ref{2.prop1} and \ref{2.Sim-Kim} 
we estimate
\begin{align}\label{est.Int1}
\int_{\Omega}b(x)|u|^{q(x)}\diff x
&\leq 2|b|_{L^{\beta(x)}(\Omega)}\big||u|^{q}\big|_{L^{\beta'(x)}(\Omega)}\notag\\
&\leq 2|b|_{L^{\beta(x)}(\Omega)}\left[1+|u|^{q^+}_{L^{q(x)\beta'(x)}(\Omega)}\right]\notag\\
&\leq 2|b|_{L^{\beta(x)}(\Omega)}\left[1+C_1\|u\|^{q^+}_{W^{1,p(x)}\left(\Omega\right)}\right]\notag\\
&\leq 2|b|_{L^{\beta(x)}(\Omega)}\left(1+C_2\|u\|^{q^+}\right),
\end{align}
where $\beta'(x):=\frac{\beta(x)}{\beta(x)-1}.$ Next, we estimate $\int_{\Omega^c}b(x)|u|^{q(x)}\diff x$. For the case $\underset{x\in\mathbb{R}^N}{\essinf}\ V(x)>0$, using assumption $(\mathcal{C})$, and Propositions~\ref{2.prop2} and \ref{2.prop.unbounded}, we have
	\begin{align*}
	\int_{\Omega^c}b(x)|u|^{q(x)}\diff x\leq b_0\int_{\Omega^c}|u|^{q(x)}\diff x&\leq b_0\left[1+|u|^{q^+}_{L^{q(x)}(\Omega^c)}\right]\notag\\
	&\leq b_0\left[1+C_3\|u\|_{W^{1,p(x)}\left(\Omega^c\right)}^{q^+}\right]\notag\\
	&\leq b_0\left(1+C_4\|u\|^{q^+}\right).\notag
	\end{align*}
	For the case $\underset{x\in\mathbb{R}^N}{\essinf}\ V(x)=0$, then by $(\mathcal{P} )$, $p$ satisfies the log-H\"older decay condition and $|E|<\infty,$ where  $E:=\{x\in\mathbb{R}^N:q(x)\ne p^\ast(x)\}$. We first notice that
	\begin{align}\label{Est.I2}
	\int_{\Omega^c}b(x)|u|^{q(x)}\diff x\leq b_0\int_{\Omega^c}|u|^{q(x)}\diff x&
	= b_0\int_{\Omega^c\setminus E}|u|^{p^\ast(x)}\diff x+b_0\int_{\Omega^c\cap E}|u|^{q(x)}\diff x\notag\\
	&
	\leq b_0\int_{\Omega^c\setminus E}|u|^{p^\ast(x)}\diff x+b_0\int_{\Omega^c\cap E}\left[1+|u|^{p^\ast(x)}\right]\diff x\notag\\
	&\leq b_0\int_{\Omega^c}|u|^{p^\ast(x)}\diff x+b_0|E|.
	\end{align}
	Then, fix $R>0$ such that $\overline{\Omega}\subset B_R$ and let $\phi\in C^\infty(\mathbb{R}^N)$ satisfy $\chi_{B_{R+1}^c}\leq \phi\leq \chi_{B_R^c}.$ Clearly, $\phi u\in C_c^\infty(\Omega^c).$ Invoking Proposition~\ref{2.prop2}, \cite[Theorem 8.3.1]{Diening} and Lemma~\ref{III.critical.imb.loc} we have
	\begin{align*}
	\int_{\Omega^c}|u|^{p^\ast(x)}\diff x&\leq \int_{B_{R+1}\setminus\overline{\Omega}}|u|^{p^\ast(x)}\diff x +\int_{\Omega^c}|\phi u|^{p^\ast(x)}\diff x\notag\\
&
\leq 2+|u|^{(p^\ast)^+}_{L^{p^\ast(x)}(B_{R+1}\setminus\overline{\Omega})}+|\phi u|^{(p^\ast)^+}_{L^{p^\ast(x)}(\Omega^c)}\notag\\
&\leq 2+C_5\|u\|^{(p^\ast)^+}_{W^{1,p(x)}(B_{R+1}\setminus\overline{\Omega})}+C_6\big||\nabla(\phi u)|\big|^{(p^\ast)^+}_{L^{p(x)}(\Omega^c)}\notag\\
&	\leq 2+C_7\|u\|^{(p^\ast)^+}_{W_V^{1,p(x)}(B_{R+1}\setminus\overline{\Omega})}+C_6\left[\big||(\nabla\phi) u|\big|_{L^{p(x)}(\Omega^c)}+\big||\phi\nabla u|\big|_{L^{p(x)}(\Omega^c)}\right]^{(p^\ast)^+}\notag\\
&	\leq 2+C_7\|u\|^{(p^\ast)^+}_{W_V^{1,p(x)}(B_{R+1}\setminus\overline{\Omega})}+C_8\left[|u|_{L^{p(x)}(B_{R+1}\setminus\overline{B_R})}+\big||\nabla u|\big|_{L^{p(x)}(\Omega^c)}\right]^{(p^\ast)^+}\notag\\
&	\leq C_9\left[1+ \|u\|^{(p^\ast)^+}\right].
	\end{align*}
Combining this and \eqref{Est.I2}, we obtain that for the case $\underset{x\in\mathbb{R}^N}{\essinf}\ V(x)=0$,	
$$\int_{\Omega^c}|u|^{q(x)}\diff x\leq C_{10}\left[1+ \|u\|^{(p^\ast)^+}\right].$$
That is, in any case we have 
$$\int_{\Omega^c}b(x)|u|^{q(x)}\diff x\leq C_{11}\left[1+ \|u\|^{(p^\ast)^+}\right].$$
Combining this and \eqref{est.Int1}, we obtain
	$$\int_{\mathbb{R}^N}b(x)|u|^{q(x)}\diff x\leq C_{12}\left[1+\|u\|^{(p^\ast)^+}\right].$$
From this, we deduce
	$$|u|_{L^{q(x)}(b,\mathbb{R}^N)}\leq 1+\left(\int_{\mathbb{R}^N}b(x)|u|^{q(x)}\diff x\right)^{\frac{1}{q^-}}\leq 1+C_{12}^{\frac{1}{q^-}}\left[1+\|u\|^{(p^\ast)^+}\right]^{\frac{1}{q^-}}$$
	in view of Proposition~\ref{2.prop2}. Thus, we obtain \eqref{1imbedding.ineq} with $C:=1+(2C_{12})^{\frac{1}{q^-}}$ and the proof is complete.
\end{proof}

\section{Proof of Lemma~\ref{III.le.est.measure}} \label{AppendixB}
\begin{proof}[Proof of Lemma~\ref{III.le.est.measure}]
	We first prove that there exists $\delta_0>0$ such that for all open set $O$ of $\mathbb{R}^N,$ we have $\nu(O)=0$ or $\nu(O)>\delta_0.$
	Indeed, by an approximation argument it is not difficult to see that \eqref{3.reverse.ineq} holds for all $\phi=\chi_A$ with any compact $A$ of $\mathbb{R}^N.$ Fix a compact $A$ of $\mathbb{R}^N$. Since $p,q\in C(\mathbb{R}^N),$ for each $x\in\mathbb{R}^N$ there exists $\epsilon=\epsilon(x)>0$ such that
	$$|p(y)-p(x)|<\frac{\delta}{4},\ |r(y)-r(x)|<\frac{\delta}{4}\quad \text{for all}\ y\in \overline{B_\epsilon(x)},$$
	where $\delta:=\underset{x\in\mathbb{R}^N}{\inf}(r(x)-p(x)).$ Thus, we have
	
	$$\left(p|_{B_\epsilon(x)}\right)^+<p(x)+\frac{\delta}{4}<r(x)-\frac{\delta}{4}<\left(r|_{B_\epsilon(x)}\right)^-,$$
	and hence,
	\begin{equation*}
	\left(r|_{B_\epsilon(x)}\right)^--\left(p|_{B_\epsilon(x)}\right)^+>\frac{\delta}{2}.
	\end{equation*}
	Since $A$ is compact, we have a finite covering of $A$ of balls $\{B_{\epsilon_i}(x_i)\}_{i=1}^n$ with
	\begin{equation}\label{3.r+p-}
	\left(r|_{B_{\epsilon_i}(x_i)}\right)^--\left(p|_{B_{\epsilon_i}(x_i)}\right)^+>\frac{\delta}{2},\quad i=1,\cdots,n.
	\end{equation}
	Fix any $i\in\{1,\cdots,n\}.$ Applying \eqref{3.reverse.ineq}, we get
	\begin{equation}\label{3.est.le.3.4}
	|\chi_{A\cap \overline{B_{\epsilon_i}(x_i)}}|_{L_\nu^{r(x)}(\mathbb{R}^N)}\leq C|\chi_{A\cap \overline{B_{\epsilon_i}(x_i)}}|_{L_\nu^{p(x)}(\mathbb{R}^N)}.
	\end{equation}
	If $\nu \left(A\cap \overline{B_{\epsilon_i}(x_i)}\right)\leq 1,$ then by invoking Proposition~\ref{2.prop2} we have
	\begin{align*}
	|\chi_{A\cap \overline{B_{\epsilon_i}(x_i)}}|_{L_\nu^{p(x)}(\mathbb{R}^N)}^{\left(p|_{B_{\epsilon_i}(x_i)}\right)^+}\leq \int_{\mathbb{R}^N}|&\chi_{A\cap \overline{B_{\epsilon_i}(x_i)}}|^{p(x)}\diff\nu=\nu \left(A\cap \overline{B_{\epsilon_i}(x_i)}\right)\\
	&= \int_{\mathbb{R}^N}|\chi_{A\cap \overline{B_{\epsilon_i}(x_i)}}|^{r(x)}\diff\nu\leq |\chi_{A\cap \overline{B_{\epsilon_i}(x_i)}}|_{L_\nu^{r(x)}(\mathbb{R}^N)}^{\left(r|_{B_{\epsilon_i}(x_i)}\right)^-}.
	\end{align*}
	Combining this with \eqref{3.est.le.3.4}, we have
	\begin{align*}
	\nu \left(A\cap \overline{B_{\epsilon_i}(x_i)}\right)^{\frac{1}{\left(r|_{B_{\epsilon_i}(x_i)}\right)^-}}&\leq |\chi_{A\cap \overline{B_{\epsilon_i}(x_i)}}|_{L_\nu^{r(x)}(\mathbb{R}^N)}\leq (C+1)|\chi_{A\cap \overline{B_{\epsilon_i}(x_i)}}|_{L_\nu^{p(x)}(\mathbb{R}^N)}\\
	&\leq (C+1) \nu \left(A\cap \overline{B_{\epsilon_i}(x_i)}\right)^{\frac{1}{\left(p|_{B_{\epsilon_i}(x_i)}\right)^+}}.
	\end{align*}
	By this and \eqref{3.r+p-} we obtain
	$\nu \left(A\cap \overline{B_{\epsilon_i}(x_i)}\right)=0$ or
	$$\nu \left(A\cap \overline{B_{\epsilon_i}(x_i)}\right)\geq \left(\frac{1}{C+1}\right)^{\frac{\left(p|_{B_{\epsilon_i}(x_i)}\right)^+\left(r|_{B_{\epsilon_i}(x_i)}\right)^-}{\left(r|_{B_{\epsilon_i}(x_i)}\right)^--\left(p|_{B_{\epsilon_i}(x_i)}\right)^+}}\geq \left(\frac{1}{C+1}\right)^{\frac{2p^+r^+}{\delta}}=:\delta_0>0.$$
	We deduce that $\nu(A)=0$ or $\nu(A)>\delta_0$ for all compact $A$ of $\mathbb{R}^N.$ Since $\nu$ is regular, we have $\nu(O)=0$ or $\nu(O)>\delta_0$ for any open $O$ of $\mathbb{R}^N.$ The conclusion of the lemma now follows from \cite[Lemma 3.3]{Bonder}.
\end{proof}


\section{Proof of Lemma~\ref{IV.compact}} \label{AppendixC}
\begin{proof}[Proof of Lemma~\ref{IV.compact}]
Let $u\in X$. In the rest of the proof, $C,C_i$ ($i\in\mathbb{N}$) are positive constants independent of $u$. By Lemma~\ref{III.critical.imb}, we have $u\in L^{p^\ast(x)}(\mathbb{R}^N)$ and 
$$|u|_{L^{p^\ast(x)}(\mathbb{R}^N)}\leq C\|u\|.$$
Invoking this inequality and Propositions~\ref{2.prop1}, \ref{2.Sim-Kim} and \ref{2.prop.unbounded}, we estimate
	\begin{align*}
	\int_{\mathbb{R}^N}a(x)|u|^{r(x)}\diff x&\leq 2|a|_{L^{\frac{p^\ast(x)}{p^\ast(x)-r(x)}}(\mathbb{R}^N)}\big||u|^{r}\big|_{L^{\frac{p^\ast(x)}{r(x)}}(\mathbb{R}^N)}\notag\\
	&\leq 2|a|_{L^{\frac{p^\ast(x)}{p^\ast(x)-r(x)}}(\mathbb{R}^N)}\left[1+|u|^{r^+}_{L^{p^\ast(x)}(\mathbb{R}^N)}\right]\\
	&\leq 2|a|_{L^{\frac{p^\ast(x)}{p^\ast(x)-r(x)}}(\mathbb{R}^N)}\left(1+C^{r^+}\|u\|^{r^+}\right).
	\end{align*}
	From this, we deduce
	$$|u|_{L^{r(x)}(a,\mathbb{R}^N)}\leq 1+\left(\int_{\mathbb{R}^N}a(x)|u|^{r(x)}\diff x\right)^{\frac{1}{r^-}}\leq 1+C_1^{\frac{1}{r^-}}\left(1+\|u\|^{r^+}\right)^{\frac{1}{r^-}}$$
	in view of Proposition~\ref{2.prop2}. Thus, we obtain
	$$|u|_{L^{r(x)}(a,\mathbb{R}^N)}\leq \left[1+(2C_1)^{\frac{1}{r^-}}\right]\|u\|,\ \forall u\in X,$$
	i.e., $X\hookrightarrow L^{r(x)}(a,\mathbb{R}^N).$ Next, we show the compactness of this imbedding. Let $u_n\rightharpoonup 0$ in $X.$ We shall show that $u_n\to 0$ in $ L^{r(x)}(a,\mathbb{R}^N)$, equivalently, $\int_{ \mathbb{R}^N}a(x)|u_n|^{r(x)}\diff x\to 0$. By Lemma~\ref{III.critical.imb.loc}, up to a subsequence we have $u_n\to 0$ a.e. in $\mathbb{R}^N.$ Next, we show that for a given $\epsilon\in (0,1)$,
	\begin{itemize}
		\item[$(1^0)$] there exists $\delta=\delta(\epsilon)>0$ such that for any measurable subset $Q\subset\mathbb{R}^N$ with $|Q|<\delta,$ we have
		$\int_{Q}a(x)|u_n|^{r(x)}\diff x<\epsilon$ for all $n\in\mathbb{N};$
		\item[$(2^0)$] there exists $R>0$ such that 	$\int_{B_R^c}a(x)|u_n|^{r(x)}\diff x<\epsilon$ for all $n\in\mathbb{N}.$
	\end{itemize}
Indeed, by the boundedness of $\{u_n\}$ in $X,$ we may assume that
	\begin{equation}\label{Proof4.1.boundedness}
	\|u_n\|\leq 1,\quad \forall n\in\mathbb{N}.
	\end{equation}
On the other hand, the assumption $a\in L^{\frac{p^\ast(x)}{p^\ast(x)-r(x)}}(\mathbb{R}^N)$ implies that there exist $\delta=\delta(\epsilon)>0$ and $R=R(\epsilon)>0$ such that 
$$\int_{B_R^c}|a(x)|^{\frac{p^\ast(x)}{p^\ast(x)-r(x)}}\diff x<\left(\frac{\epsilon}{2(1+C^{r^+})}\right)^{\left(\frac{p^\ast}{p^\ast-r}\right)^+}<1,$$
and for any measurable subset $Q\subset\mathbb{R}^N$ with $|Q|<\delta$
we have
$$\int_{Q}|a(x)|^{\frac{p^\ast(x)}{p^\ast(x)-r(x)}}\diff x<\left(\frac{\epsilon}{2(1+C^{r^+})}\right)^{\left(\frac{p^\ast}{p^\ast-r}\right)^+}.$$
Let $Q$ be a measurable subset $Q\subset\mathbb{R}^N$ with $|Q|<\delta$. We have that 
	\begin{align*}
	\int_{Q}a(x)|u_n|^{r(x)}\diff x&\leq 2|a|_{L^{\frac{p^\ast(x)}{p^\ast(x)-r(x)}}(Q)}\big||u_n|^{r}\big|_{L^{\frac{p^\ast(x)}{r(x)}}(Q)}\notag\\
	&\leq 2|a|_{L^{\frac{p^\ast(x)}{p^\ast(x)-r(x)}}(Q)}\left[1+|u_n|^{r^+}_{L^{p^\ast(x)}(Q)}\right]\notag\\
	&\leq 2\left(\int_{Q}|a|^{\frac{p^\ast(x)}{p^\ast(x)-r(x)}}\diff x\right)^{\frac{1}{\left(\frac{p^\ast}{p^\ast-r}\right)^+}}\left(1+C^{r^+}\|u_n\|^{r^+}\right)\notag\\
	&<\epsilon,\ \forall n\in\mathbb{N}. 
	\end{align*}
Thus, we have proved $(1^0).$ Similarly, we have that	
	\begin{align*}
	\int_{B_R^c}a(x)|u_n|^{r(x)}\diff x&\leq 2|a|_{L^{\frac{p^\ast(x)}{p^\ast(x)-r(x)}}(B_R^c)}\big||u_n|^{r}\big|_{L^{\frac{p^\ast(x)}{r(x)}}(B_R^c)}\notag\\
	&\leq 2|a|_{L^{\frac{p^\ast(x)}{p^\ast(x)-r(x)}}(B_R^c)}\left[1+|u_n|^{r^+}_{L^{p^\ast(x)}(B_R^c)}\right]\notag\\
	&\leq 2\left(\int_{B_R^c}|a|^{\frac{p^\ast(x)}{p^\ast(x)-r(x)}}\diff x\right)^{\frac{1}{\left(\frac{p^\ast}{p^\ast-r}\right)^+}}\left(1+C^{r^+}\|u_n\|^{r^+}\right)\notag\\
	&<\epsilon,\ \forall n\in\mathbb{N}. 
	\end{align*}
That is, we have proved $(2^0)$. We now apply the Vitali convergence theorem to get that $\int_{ \mathbb{R}^N}a(x)|u_n|^{r(x)}\diff x\to 0$ and hence, the proof is complete. 
\end{proof}

\subsection*{Acknowledgements}
The first author was supported by Nong Lam University, Ho Chi Minh City (University project, 2018), the second author was supported by the Basic Science Research Program through the National Research Foundation of Korea (NRF) funded by the Ministry of Education (NRF-2016R1D1A1B03935866), and the third author was supported by a National Research Foundation of Korea Grant funded by
the Korean Government (MEST) (NRF-2018R1D1A3A03000678).

\end{document}